\documentclass[a4paper, DIV=12, 11pt]{scrartcl}
\usepackage[latin1]{inputenc}
\usepackage{amsmath}
\usepackage{amsfonts}
\usepackage{amssymb}
\usepackage{amsthm}
\usepackage{graphicx}
\usepackage{thmtools}
\usepackage{bbm}
\usepackage{hyperref}
\usepackage{enumitem}
\usepackage{todonotes}

\usepackage{tikz}
\usetikzlibrary{matrix}

\declaretheoremstyle[headfont=\normalfont]{normalhead}
\newtheorem{lemma}{Lemma}[section]
\newtheorem{theorem}[lemma]{Theorem}
\newtheorem{proposition}[lemma]{Proposition}
\newtheorem{corollary}[lemma]{Corollary}
\newtheorem{definition}[lemma]{Definition}
\newtheorem{remark}[lemma]{Remark}

\newtheorem{maintheorem}{Theorem}

\newcommand{\R}{\mathbb{R}}
\newcommand{\N}{\mathbb{N}}
\newcommand{\C}{\mathbb{C}}
\newcommand{\Val}{\mathrm{Val}}

\newcommand{\lip}{\mathrm{Lip}}
\newcommand{\Dense}{\mathrm{Dens}}

\newcommand{\K}{{\mathcal K}}
\newcommand{\hm}{\mathcal{H}^{n-1}}
\newcommand{\unitsphere}{\mathrm{S}^{n-1}}

\newcommand{\GL}{\mathrm{GL}}
\newcommand{\Kfin}{K-\mathrm{fin}}
\newcommand{\SOnFin}{\SO(n)-\mathrm{fin}}
\newcommand{\SO}{\mathrm{SO}}
\renewcommand{\P}{\mathbb{P}_+(V^*)}
\newcommand{\Sym}{\mathrm{Sym}}

\DeclareMathOperator{\Div}{\operatorname{div}}

\setkomafont{title}{\normalfont\Large}

\author{Andrea Colesanti, Jonas Knoerr, Daniele Pagnini}
\title{The homogeneous decomposition\\
of dually translation invariant valuations\\
on Lipschitz functions on the sphere}
\date{}

\newcommand{\Addresses}{{
		\bigskip
		\footnotesize
		\textsc{Dipartimento di Matematica e Informativa ``U. Dini" Universit\`a degli Studi di Firenze, Viale Morgagni 67/A - 50134, Firenze, Italy}\par \nopagebreak
		\textit{E-mail address}: \texttt{andrea.colesanti@unifi.it}\\
		
		\textsc{Institute of Discrete Mathematics and Geometry, Technische Universität Wien, Wiedner Hauptstrasse 8-10, 1040 Wien, Austria}\par\nopagebreak
		\textit{E-mail address}: \texttt{jonas.knoerr@tuwien.ac.at}\\	
		
		\textsc{Dipartimento di Matematica e Informativa ``U. Dini" Universit\`a degli Studi di Firenze, Viale Morgagni 67/A - 50134, Firenze, Italy}\par \nopagebreak
		\textit{E-mail address}: \texttt{danielepagnini91@gmail.com}\\
		
	}}

\makeindex
\begin{document}
\maketitle
\begin{abstract}
	We show that every continuous and dually translation invariant valuation on the space of Lipschitz functions on the unit sphere of $\R^n$, $n\ge2$, can be decomposed uniquely into a sum of homogeneous valuations of degree $0$, $1$ and $2$. In particular, there does not exist any non-trivial, continuous and dually translation invariant valuation which is homogeneous of degree $3$ or higher. For the space of those of degree $0$, $1$ and $2$ we provide a description of a dense subspace.
\end{abstract}
	\section{Introduction}
		Let $\mathcal{K}(\R^n)$ denote the space of convex bodies in $\R^n$, that is, the set of all non-empty, convex and compact subsets of $\R^n$ equipped with the Hausdorff metric. A map $\mu:\mathcal{K}(\R^n)\rightarrow\C$ is called a \emph{valuation} if 
		\begin{align*}
			\mu(K\cup L)+\mu(K\cap L)=\mu(K)+\mu(L)
		\end{align*}
		for all $K,L\in\mathcal{K}(\R^n)$ such that $K\cup L\in\mathcal{K}(\R^n)$. The class of valuations includes many important geometric functionals, like the Lebesgue measure, surface area measure and Euler characteristic, as well as the intrinsic volumes and, more generally, the mixed volumes. It thus plays an important role in convex and integral geometry. We refer to \cite[Chapter 6]{SchneiderConvexbodiesBrunn2014} for a detailed account of the theory of valuations on convex bodies. In particular, we are interested in the following decomposition of the space $\Val(\R^n)$ of continuous and translation invariant valuations on $\mathcal{K}(\R^n)$: if $\Val_k(\R^n)$ denotes the subspace of $\Val(\R^n)$ of \emph{$k$-homogeneous} valuations, that is, all $\mu\in\Val(\R^n)$ satisfying $\mu(tK)=t^k\mu(K)$ for $t\ge 0$, $K\in\K(\R^n)$, then the following holds.
		\begin{theorem}[McMullen decomposition, \cite{McMullenValuationsEulerType1977}]\label{theorem_McMullen_decomposition} Let $\mu\in\Val(\R^n)$; then there exist unique $\mu_0,\dots,\mu_n$, with $\mu_k\in\Val_k(\R^n)$ for every $k=0,\dots,n$, such that 
			$$
			\mu=\mu_0+\dots+\mu_n.
			$$
		\end{theorem}	
		In other words, we have the direct sum decomposition								
		\begin{align*}
		\Val(\R^n)=\bigoplus\limits_{k=0}^n\Val_k(\R^n).
		\end{align*}
		It is easy to see that all the elements in $\Val_0(\R^n)$ are constant valuations. Moreover, $\Val_n(\R^n)$ is spanned by the Lebesgue measure, thanks to a result due to Hadwiger \cite{HadwigerVorlesungenuberInhalt1957}. While there exists a complete description of the elements of $\Val_{n-1}(\R^n)$ given by McMullen \cite{McMullenContinuoustranslationinvariant1980}, for $1\le k\le n-2$ we only have results, mostly based on Alesker's Irreducibility Theorem \cite{AleskerDescriptiontranslationinvariant2001}, that identify dense subspaces of $ \Val_k(\R^n) $. These include the space spanned by mixed volumes as well as the space of smooth valuations, for which many different descriptions are known. Here we consider $\Val(\R^n)$ with the topology induced by the semi-norm
		\begin{align*}
			\|\mu\|:=\sup_{K\subset B_1(0)}|\mu(K)|,\quad\text{for }\mu\in\Val(\R^n),
		\end{align*}
		where $B_1(0)$ denotes the unit ball centered at the origin. The homogeneous decomposition implies that $\|\cdot\|$ is actually a norm, which turns $\Val(\R^n)$ into a Banach space.\\
		
		The notion of valuation can be extended from sets to families $X$ of real-valued functions: we call a map $\mu:X\rightarrow \C$ a (complex-valued) \emph{valuation} if
		\begin{align*}
			\mu(f\vee h)+\mu(f\wedge h)=\mu(f)+\mu(h)
		\end{align*}
		for all $f,h\in X$ for which the pointwise maximum $f\vee h$ and minimum $f\wedge h$ belong to $X$. When $X$ denotes the set of indicator functions of convex bodies, this corresponds to the notion of valuations on convex bodies. Although the study of valuations on functions only started rather recently, there now exists a large body of research concerning valuations on many classical function spaces, for example $L_p$ and Orlicz spaces \cite{TsangValuations$Lp$spaces2010,Tsang2,Kone}, Sobolev spaces and spaces of functions of bounded variations  \cite{LudwigFisherinformationmatrix2011,LudwigValuationsSobolevspaces2012,LudwigCovariancematricesvaluations2013,Ma,Wang}, spaces connected to convexity: convex, log-concave and quasi-concave functions \cite{AleskerValuationsconvexfunctions2019,Cavallina-Colesanti,Colesanti-Lombardi,Colesanti-Lombardi-Parapatits,CLM-Minkowskivaluationsonconvexfunctions,CLM-Valuationsonconvexfunctions,ColesantiEtAlhomogeneousdecompositiontheorem2020,ColesantiEtAlHadwigertheoremconvex2020,CLM-Hessianvaluations,CLM-HadwigerII,CLM-HadwigerIII,CLM-HadwigerIV,KnoerrSmoothvaluationsconvex2020,Knoerrsupportduallyepi2021,Knoerr-2022,Knoerr-Hofstaetter-2022,Li-Ma,Mussnig-2019,Mussnig-2021}, definable functions \cite{BaryshnikovEtAlHadwigersTheoremdefinable2013} and general Banach lattices \cite{TradaceteVillanuevaValuationsBanachLattices2018}. It is also worth mentioning the articles \cite{Villanueva,Tradacete-Villanueva-2017,Tradacete-Villanueva-2018}, which are related to valuations on the space $C(\unitsphere)$. 
		
		We will be interested in valuations on the space $\lip(\unitsphere)$ of Lipschitz functions on the unit sphere $\unitsphere$, $n\ge 2$, where some classification results were obtained in \cite{ColesantiEtAlclassinvariantvaluations2020,ColesantiEtAlContinuousvaluationsspace2021} by Tradacete, Villanueva, as well as the first and last named authors. We will equip this space with a topology that is also used in these articles (see Section \ref{section:lip_basics}), and we will impose the following invariance property: a valuation $\mu:\lip(\unitsphere)\rightarrow\C$ is called \emph{dually translation invariant} (\emph{$\Lambda$-invariant} in \cite{ColesantiEtAlclassinvariantvaluations2020}) if it is invariant under the addition of restrictions of linear functions to the unit sphere, that is,
		\begin{align*}
			\mu(f+\lambda|_{\unitsphere})=\mu(f)\quad\text{for all }f\in\lip(\unitsphere),\,\lambda:\R^n\rightarrow\R\text{ linear}.
		\end{align*}
		This notion has a geometric interpretation: recall that for any convex body $K\in\mathcal{K}(\R^n)$ one may define its \emph{support function} $h_K:\R^n\rightarrow\R$ by
		\begin{align*}
			h_K(y)=\sup_{x\in K}\langle y,x\rangle,\quad\text{for }y\in\R^n.
		\end{align*}
		Note that $h_K$ is convex and $1$-homogeneous, thus it is in particular Lipschitz continuous. Therefore, it is natural to consider support functions as Lipschitz continuous functions defined on the unit sphere. This implies that every valuation $\mu:\lip(\unitsphere)\rightarrow\C$ defines a valuation on $\mathcal{K}(\R^n)$ by $K\mapsto \mu(h_K)$ (see \cite{ColesantiEtAlclassinvariantvaluations2020}, Lemma 2.7), hence dually translation invariant valuations on $\lip(\unitsphere)$ correspond to translation invariant valuations on convex bodies. 
		
		Let $\Val(\lip(\unitsphere))$ denote the space of continuous and dually translation invariant valuations on $\lip(\unitsphere)$. Due to the relation to translation invariant valuations on convex bodies, this space also admits a homogeneous decomposition, as shown in \cite{ColesantiEtAlclassinvariantvaluations2020}, Theorem 4.1  (note that the result proved in \cite{ColesantiEtAlclassinvariantvaluations2020} holds for real-valued valuations, but it immediately extends to the complex-valued case). Let $k\in\{0,\ldots,n\}$ and let $\Val_k(\lip(\unitsphere))$ be the subspace of $\Val(\lip(\unitsphere))$ of $k$-homogeneous valuations, that is, all $\mu\in\Val(\lip(\unitsphere))$ that satisfy $\mu(tf)=t^k\mu(f)$ for $t\ge 0$, $f\in \lip(\unitsphere)$; then
		\begin{align*}
			\Val(\lip(\unitsphere))=\bigoplus_{k=0}^n\Val_k(\lip(\unitsphere)).
		\end{align*}
		However, not all translation invariant valuations on convex bodies can be obtained from valuations on $\lip(\unitsphere)$ in this way, as shown by the next result. To state it, recall that by Rademacher's theorem, any $f\in\lip(\unitsphere)$ is $\hm$-a.e. differentiable, where $ \hm $ is the $ (n-1)$-dimensional Hausdorff measure; if $f$ is differentiable at $x\in\unitsphere$, we will denote its spherical gradient by $\nabla f(x)$.
		\begin{theorem}[\cite{ColesantiEtAlclassinvariantvaluations2020}, Theorem 1.2]
			\label{theorem:HadwigerLipschitz}
			A valuation $\mu\in\Val(\lip(\unitsphere))$ is rotation invariant if and only if there exist constants $c_0,c_1,c_2\in\C$ such that
			\begin{align*}
				\mu(f)=c_0+c_1\int_{\unitsphere}f(x)d\hm(x)+c_2\int_{\unitsphere}\left[(n-1)f(x)^2-|\nabla f(x)|^2\right]d\hm(x),
			\end{align*}
			for every $f\in\lip(\unitsphere)$. In particular, there does not exist any non-trivial, rotation invariant valuation in $ \Val(\lip(\unitsphere)) $ that is homogeneous of degree $k\ge 3$.
		\end{theorem}
		This result differs from Hadwiger's characterization of continuous, rotation and translation invariant valuations on $\K(\R^n)$ \cite{HadwigerVorlesungenuberInhalt1957}. The space of these valuations is spanned by the $ n+1 $ intrinsic volumes; in particular, there exists a non-trivial rotation invariant valuation in $\Val_k(\R^n)$ for every $0\le k\le n$. This naturally leads to the question whether there exist any non-trivial valuations in $\Val(\lip(\unitsphere))$ that are homogeneous of degree $k\ge 3$. Our main result shows that this is not the case.
		\begin{maintheorem}
			\label{maintheorem:decomposition}
			Let $ n,k\in\N $, with $n\ge 3$ and $3\leq k\leq n$. If $\mu\in\Val_k(\lip(\unitsphere))$ is a $k$-homogeneous valuation, then $\mu=0$. In particular
			\begin{align*}
				\Val(\lip(\unitsphere))=\bigoplus_{k=0}^2\Val_k(\lip(\unitsphere)).
			\end{align*}
		\end{maintheorem}

		This result is also related to the following question: which $k$-homogeneous valuations in $\Val(\R^n)$ extend to $\lip(\unitsphere)$? Theorem \ref{theorem:HadwigerLipschitz} tells us that this is not the case for the intrinsic volumes for $k\ge 3$. The proof given in \cite{ColesantiEtAlclassinvariantvaluations2020} shows that no extension can be continuous, by constructing a convergent sequence in $\lip(\unitsphere)$  such that the values of the extension on these functions do not converge. The construction of this sequence heavily depends on the properties of the $k$th intrinsic volume. While it may be possible to generalize this argument to arbitrary valuations in $\Val_k(\R^n)$, $k\ge 3$, we will tackle this problem using some tools from representation theory. 
		
		Here is a synthetic description of our argument. First, note that we may consider the elements of $\lip(\unitsphere)$ as restrictions of $1$-homogeneous functions defined on $\R^n$: this induces a natural operation of the general linear group on this space. Thus, $\Val(\lip(\unitsphere))$ turns into a representation of $\GL(n,\R)$, and we will show that this operation is continuous with respect to a suitable topology on $\Val(\lip(\unitsphere))$. The map \begin{align*}
		T:\Val(\lip(\unitsphere))&\rightarrow\Val(\R^n)\\
		\mu&\mapsto \left[K\mapsto \mu(h_K)\right]
		\end{align*}
		is then $\GL(n)$-equivariant, injective and continuous. Consider the subspaces $\Val_k^\pm(\lip(\unitsphere))$ of even and odd valuations, where we call $\mu\in\Val_k(\lip(\unitsphere))$ even (respectively, odd) if $\mu(f(-\cdot))= \mu(f)$ (respectively $\mu(f(-\cdot))=-\mu(f)$), for all $f\in\lip(\unitsphere)$. Using results due to Alesker \cite{AleskerDescriptiontranslationinvariant2001} and Alesker, Bernig and Schuster  \cite{AleskerEtAlHarmonicanalysistranslation2011}, we will show that the image of $\Val^\pm_k(\lip(\unitsphere))$ under $T$ is $0$ unless every even/odd $\SO(n)$-finite valuation in $\Val_k(\R^n)$ extends to an element of $\Val^{\pm}_k(\lip(\unitsphere))$. This reduces the problem to valuations that admit a very simple and explicit description, for which we can generalize the original argument in \cite{ColesantiEtAlclassinvariantvaluations2020}. More precisely, we will see that we only need to prove that one additional special valuation cannot be extended, see Section \ref{section:proofOddCase}.\\
		
		Our approach also leads to a description of a dense subspace of $\Val(\lip(\unitsphere))$. Recall that if $(\pi,F)$ is a continuous representation of $\GL(n,\R)$ on a locally convex vector space $F$, then a vector $v\in F$ is called \emph{smooth} if the map
		\begin{align*}
			\GL(n,\R)&\rightarrow F\\
			g&\mapsto \pi(g)\mu
		\end{align*}
		is infinitely differentiable. If $F$ is complete, then the set $F^{sm}$ of smooth vectors is dense in $F$. We will see that there is a one-to-one correspondence between the smooth vectors of $\Val_k(\R^n)$ and $\Val_k(\lip(\unitsphere))$ for $k=0,1,2$, see Corollary \ref{corrollary:ExtensionSmoothVal1Hom} and Corollary \ref{corrollary:ExtensionSmoothVal2Hom}. In this way we obtain explicit integral representations for these valuations. In the $1$-homogeneous case we prove the following result.  
				\begin{maintheorem}
			\label{maintheorem:1homSmoothValuations}
			A valuation $\mu\in\Val_1(\lip(\unitsphere))$ is smooth if and only if there exists $\varphi\in C^\infty(\unitsphere)$ such that
			\begin{align*}
				\mu(f)=\int_{\unitsphere}\varphi(x)f(x)d\hm(x)\quad\text{for all }f\in\lip(\unitsphere),
			\end{align*}
			and $\varphi$ is orthogonal to the restriction of linear functions to $\unitsphere$.
		\end{maintheorem}
		A similar description of smooth valuations was also given by Alesker in the appendix of \cite{BergEtAlLogconcavityproperties2018} in the context of smooth valuations in $\Val_1(\R^n)$. 
		
		For homogeneous valuations of degree $2$, we obtain another result, which for support functions gives a new description of smooth translation invariant valuations on convex bodies of degree $2$. Let us set $\bar{\nabla}f(x):=\nabla f(x)+f(x)x\in\R^n$, for $f\in \lip(\unitsphere)$ for all $x\in \unitsphere$ such that $f$ is differentiable at $x$. This coincides, $\hm$-a.e. on $\unitsphere$, with the Euclidean gradient of the $1$-homogeneous extension of $f$. Our result is the following.
		\begin{maintheorem}
			\label{maintheorem:2homSmoothValuations}
			A valuation $\mu\in\Val_2(\lip(\unitsphere))$ is smooth if and only if there exists a function $\varphi\in C^\infty(\unitsphere,\Sym^2(\C^n))$, with values in the space $\Sym^2(\C^n)$ of symmetric $n\times n$-matrices, such that
			\begin{align*}
				\mu(f)=\int_{\unitsphere}\left\langle \varphi(x)\bar{\nabla}f(x),\bar{\nabla}f(x)\right\rangle d\hm(x),\quad\text{for all }f\in\lip(\unitsphere),
			\end{align*}
			and $\varphi$ has the following property:
			\begin{align}
				\label{eq:conditionDensity2Hom}
				\int_{\unitsphere}\varphi(x)\bar{\nabla}f(x)d\hm(x)=0\quad\text{for all }f\in\lip(\unitsphere).
			\end{align}
		\end{maintheorem}
		Note that using integration by parts, \eqref{eq:conditionDensity2Hom} reduces to a system of linear partial differential equations of order $1$ for $\varphi$, see Remark \ref{pde system}.
		
	\section{Notations and background}
		\label{section:NotationsBackground}	
		\subsection{Notations} 

 We work in the $n$-dimensional Euclidean space $\R^n$, $n\ge1$. The Euclidean norm and the scalar product are denoted by $|\cdot|$ and $\langle \cdot,\cdot\rangle$, respectively. When dealing with complex numbers, we will denote the norm in $\C$ by $|\cdot|$ as well (the ambient space will always be clear from the context). For $k\in\{0,1,\dots,n\}$, ${\mathcal H}^k$ denotes the $k$-dimensional Hausdorff measure in $\R^n$.
	
		\subsection{Some notions and results from representation theory}
		\label{section:representationTheory} 
			In this section we recall some facts from representation theory of real reductive groups. We refer to \cite{WallachRealreductivegroups1988} for further details and to \cite{SepanskiCompactLiegroups2007} for a background on the representation theory of compact Lie groups.
			
			By a (continuous) \emph{representation} $(\pi,G,F)$ of a Lie group $G$ on a locally convex complex vector space $F$, we mean a group homomorphism $\pi:G\rightarrow GL(F)$ into the automorphism group of $F$ such that the map
			\begin{align*}
				G\times F&\rightarrow F\\
				(g,v)&\mapsto \pi(g)v
			\end{align*}
			is continuous. In this case, $F$ is also called a \emph{$G$-module}, and if $(\pi_i,G,F_i)$ $i=1,2$ are two $G$-modules, then a continuous linear map $T:F_1\rightarrow F_2$ that commutes with the action of $G$ is called a \emph{$G$-morphism}.
			\begin{definition}
				Let $(\pi,G,F)$ be a continuous representation of a Lie group $G$ on a locally convex complex vector space $F$. A vector $v\in F$ is called \emph{$G$-smooth} if the map
				\begin{align*}
					G\rightarrow& F\\
					g\mapsto& \pi(g)v
				\end{align*}
				is an infinitely differentiable map.
			\end{definition}
			It is well known (see for example \cite{WallachRealreductivegroups1988}) that the space $F^{sm}$ of smooth vectors is a dense subspace of $F$ if $F$ is complete. If $F$ is a Fr\'echet space, then $F^{sm}$ carries a natural Fr\'echet topology that is finer than the subspace topology. Furthermore the representation $(\pi,G,F^{sm})$ is a continuous representation with respect to this topology and every vector is $G$-smooth. Such a representation is also called a \emph{smooth representation}.
			
			Recall also that a continuous representation $(\pi,G,F)$ is called (topologically) \emph{irreducible} if it does not admit a proper, closed, $G$-invariant subspace.\\
			
			Let us now focus on real reductive groups. We refer to \cite{WallachRealreductivegroups1988} for the precise definition. For our purposes, it is enough to only consider the real reductive group $G=\GL(n,\R)$; however, we will state the following results in the general setting. Let $K$ be a maximally compact subgroup of $G$, for example $\mathrm{O}(n)$ for $G=\GL(n,\R)$. A vector $v\in F$ is called \emph{$K$-finite} if its $K$-orbit spans a finite-dimensional subspace of $F$. If $F$ is a complete locally convex vector space, then the subspace $F^{\Kfin}$ of $K$-finite vectors is a dense subspace (see \cite{SepanskiCompactLiegroups2007}, Theorem 3.46).

			If $(\pi,K,E_\pi)$ is an irreducible representation of $K$ on a (necessarily finite-dimensional) complex vector space $E_\pi$, let $F[\pi]$ be the closure of the sum of all finite-dimensional subspaces of $F$ that are isomorphic to $E_\pi$ as a $K$-module. Let $[\pi]$ denote the equivalence class of $(\pi,K,E_\pi)$ with respect to $K$-morphisms. Then $F[\pi]$ is called the \emph{$[\pi]$-isotypical component} of $F$. Note that $F[\pi]\cap F[\pi']=0$ if $[\pi]$ and $[\pi']$ are two irreducible representations of $K$ which are inequivalent. 
			We call a representation $F$ \emph{admissible} if $F[\pi]$ is finite-dimensional for every irreducible representation $[\pi]$. In this case, the dimension of $F[\pi]$ is an integer multiple of the dimension of $E_\pi$, which we call the \emph{multiplicity} of $[\pi]$ in $F$. If all multiplicities are either $0$ or $1$, then the representation is called \emph{multiplicity-free}.	
			
		Set $|g|:=\max(\|g\|,\|g^{-1}\|)$ for $g\in\GL(n,\R)$, where $\|g\|$ denotes the operator norm of $g\in\GL(n,\R)$ with respect to some (fixed) scalar product.
			\begin{definition}
				Let $(\pi,\GL(n,\R),F)$ be a smooth representation of $\GL(n,\R)$ on a  complex Fr\'echet space $F$. We say that this representation has moderate growth if for every continuous semi-norm $\lambda$ on $F$ there exist a semi-norm $\nu_\lambda$ on $F$ and $d_\lambda\in\R$ such that
				\begin{align*}
					\lambda(\pi(g)v)\le |g|^{d_\lambda}\nu_\lambda(v),\quad\forall v\in F, g\in\GL(n,\R).
				\end{align*}
			\end{definition}
			It is well known (see for example \cite{WallachRealreductivegroups1988}, Lemma 11.5.1) that the space of smooth vectors of a Banach representation equipped with its natural topology is a smooth representation of moderate growth.
		
			Recall also that a representation $(\pi,G,F)$ has \emph{finite length} if there exists a finite filtration
			\begin{align*}
				0=F_0\subset F_1\subset \dots F_N=F
			\end{align*}
			of closed, $G$-invariant subspaces such that $F_{i+1}/F_i$ is irreducible for all $i=0,\dots, N-1$.
			
		We are now able to state the following important result concerning smooth Fr\'echet representations of moderate growth.
		\begin{theorem}[Casselman-Wallach \cite{CasselmanCanonicalextensionsHarish1989}]
			\label{theorem:CaseelmanWallach}
			Let $(\pi_1,G,F_1),(\pi_2,G,F_2)$ be two smooth representations of moderate growth of a real reductive group $G$ on complex Fr\'echet spaces $F_1,F_2$. If $F_2$ is in addition admissible and of finite length, then the image of every continuous $G$-morphism $T:F_1\rightarrow F_2$ is closed.
		\end{theorem}
		To lighten the notation, from hereon we will omit $\pi$ and $G$ from the representation $(\pi,G,F)$ whenever their meaning is clear from the context. In this paper, the role of the representation $F_1$ mentioned in the previous theorem will always be played by a representation as those involved in the following result.
		\begin{proposition}[Alesker \cite{Aleskermultiplicativestructurecontinuous2004}, Proposition A.6.]
			\label{proposition_smooth_representation_moderate_growth}
			Let $\GL(n,\R)$ act transitively on a compact smooth manifold $X$. Let $\mathcal{E}$ be a finite-dimensional $\GL(n,\R)$-equivariant vector bundle over $X$. Then the smooth representation $C^\infty(X,\mathcal{E})$ of all smooth sections of $\mathcal{E}$ has moderate growth as a $\GL(n,\R)$-module.
		\end{proposition}
			
		\subsection{Translation invariant valuations on convex bodies}
We will need several notions and results from the theory of convex bodies; in particular we will be concerned with valuations on convex bodies. Our standard reference is the monograph \cite{SchneiderConvexbodiesBrunn2014} by Schneider. 

We denote by $\K(\R^n)$ the family of non-empty, compact and convex subsets of $\R^n$, that is, {\em convex bodies}. $\K(\R^n)$ is a complete metric space, once we equip it with the Hausdorff metric.

Given  $K\in\K(\R^n)$, $h_K\colon \R^n\to\R$ denotes the \emph{support function} of $K$, which is defined as
$$
h_K(x)=\sup_{y\in K} \langle x,y\rangle,\quad\forall\ y\in \R^n,
$$
where $ \langle\cdot,\cdot\rangle $ denotes the standard scalar product of $ \R^n $. The support function of a convex body is $1$-homogeneous and convex. Vice versa, for every $h\colon \R^n\to\R$, $1$-homogeneous and convex, there exists a unique $K\in\K(\R^n)$ such that $h=h_K$. We also point out the following relation:
$$
\delta_H(K,L)=\|h_K-h_L\|_{L^\infty (\unitsphere)},\quad\forall\, K,L\in\K(\R^n),
$$ 
where $\delta_H$ denotes the Hausdorff distance. In what follows, given $K\in\K(\R^n)$ we will mainly consider the restriction of $h_K$ to the unit sphere $\unitsphere$; for simplicity, this will still be denoted by $ h_K $. 

\medskip

A (complex-valued) {\em valuation} is a functional $\mu\colon\K(\R^n)\to\C$ such that 
$$
\mu(K\cup L)+\mu(K\cap L)=\mu(K)+\mu(L)
$$
for every $K,L\in\K(\R^n)$ such that $K\cup L\in\K(\R^n)$. Note that $ K\cap L\in\K(\R^n) $ automatically in this case. A valuation $\mu$ is called {\em translation invariant} if 
$$
\mu(K+x_0)=\mu(K)
$$
for every $K\in\K(\R^n)$ and for every $x_0\in \R^n$. The space of continuous and translation invariant valuations on $\R^n$ will be denoted by $\Val(\R^n)$. For $k\in\{0,\dots,n\}$, a valuation $\mu$ is called \emph{$k$-homogeneous} if 
$$
\mu(\lambda K)=\lambda^k\mu(K)
$$
for every $K\in\K(\R^n)$ and for every $\lambda\ge0$. Here of course 
$$
\lambda K=\{\lambda x\colon x\in K\}
$$
is the standard dilation. $\Val_k(\R^n)$ denotes the space of continuous, translation invariant and $k$-homogeneous valuations on $\K(\R^n)$. 

We call a valuation $\mu\in\Val(\R^n)$ \emph{even} if $\mu(-K)=\mu(K)$ for all $K\in\K(\R^n)$, and \emph{odd} if $\mu(-K)=-\mu(K)$ for all $K\in\K(\R^n)$. If we denote the subspaces of even and odd valuations in $\Val_k(\R^n)$ by $\Val^+_k(\R^n)$ and $\Val_k^-(\R^n)$ respectively, we obviously obtain the decomposition 
\begin{align*}
	\Val_k(\R^n)=\Val_k^+(\R^n)\oplus\Val_k^-(\R^n).
\end{align*}
Note that these spaces are invariant under the natural action of $\GL(n,\R)$ on $\Val(\R^n)$ given by $[\pi(g)\mu](K)=\mu(g^{-1}K)$ for $g\in\GL(n,\R)$, $\mu\in\Val(\R^n)$, and $K\in\K(\R^n)$.\\

The next result is a characterization of the elements of $\Val_n(\R^n)$.

\begin{theorem}[Hadwiger \cite{HadwigerVorlesungenuberInhalt1957}]
	\label{theorem:Hadwiger}
	If $\mu$ is an $n$-homogeneous valuation, then $\mu$ is a multiple of the Lebesgue measure. In particular, every $n$-homogeneous valuation is even.
\end{theorem}
Note that the Lebesgue measure defines a simple valuation, where $\mu\in\Val(\R^n)$ is called \emph{simple} if it vanishes on convex bodies of dimension $(n-1)$ or lower. By a result of Klain \cite{KlainshortproofHadwigers1995}, the volume is, up to constants, the only even valuation with this property. A similar result holds for odd valuations.
\begin{theorem}[Schneider \cite{SchneiderSimplevaluationsconvex1996}]
	\label{thm:SchneiderSimple}
	Every odd, simple valuation $\mu\in\Val(\R^n)$ is of the form
	\begin{align*}
		\mu(K)=\int_{\unitsphere}\phi(x)dS_{n-1}(K,x),
	\end{align*}
	where $dS_{n-1}(K,\cdot)$ is the surface area measure of $K\in\K(\R^n)$ and $\phi\in C(\unitsphere)$ is an odd function. In particular, any odd, simple valuation is homogeneous of degree $n-1$.
\end{theorem}

This has the following important implication.
\begin{lemma}
	\label{lemma:InjectivitySchneiderEmbedding}
	Let $\mu\in\Val_k^{-}(\R^n)$ be an odd valuation and suppose that $\mu$ vanishes on all convex bodies of dimension up to $k+1$. Then $\mu=0$.
\end{lemma}
\begin{proof}
	This was shown by Alesker in \cite{AleskerDescriptiontranslationinvariant2001}, Proposition 2.6, using slightly different terminology. We will repeat the argument for the convenience of the reader. Suppose that $\mu\in\Val_{k}^-(\R^n)$ satisfies the hypothesis and consider its restriction to a $k+2$-dimensional subspace. Then $\mu$ is a simple odd valuation on this subspace by assumption, and thus homogeneous of degree $k+1$ by Theorem \ref{thm:SchneiderSimple}, which is only possible if the restriction vanishes. In other words, the restriction of $\mu$ to any $k+2$-dimensional subspace vanishes, that is, $\mu$ vanishes on all convex bodies of dimension up to $k+2$. By an induction argument, $\mu$ has to vanish identically.
\end{proof}
In general, no complete characterization of the elements in $\Val(\R^n)$ is known, however, we have the following description of $\Val(\R^n)$ as a representation of $\GL(n,\R)$. 
\begin{theorem}[Alesker's Irreducibility Theorem \cite{AleskerDescriptiontranslationinvariant2001}]
	\label{theorem_Alesker_irreducibility_theorem}
	The natural representation of $\GL(n,\R)$ on $\Val^\pm_k(\R^n)$ is topologically irreducible.
\end{theorem}
It follows from the existence of the Klain and Schneider embeddings, see \cite[Section 9]{AleskerIntroductiontheoryvaluations2018}, that $\Val^\pm_k(\R^n)$ is an admissible representation of $\GL(n,\R)$. Moreover, the work of Harish-Chandra (see for example \cite[Theorem 5]{HarishChandraRepresentationssemisimpleLie1953}) implies that for an admissible representation of $\GL(n,\R)$ on a Banach space $E$, irreducibility of the representation $E$ is equivalent to the irreducibility of the subspace $E^{sm}$ of smooth vectors with respect to its natural topology. In particular, Alesker's Irredicibility Theorem also holds for smooth valuations:
\begin{corollary}
	\label{corollary_Alesker_irreducibility_theorem_smooth_case}
	The natural representation of $\GL(n,\R)$ on $\Val^\pm_k(\R^n)^{sm}$ is topologically irreducible with respect to the natural Fr\'echet topology.
\end{corollary}

A description of the $\SO(n)$-isotypical components of these spaces was obtained in \cite{AleskerEtAlHarmonicanalysistranslation2011}. For our purposes, the following result will be sufficient.
\begin{theorem}[Alesker-Bernig-Schuster \cite{AleskerEtAlHarmonicanalysistranslation2011}]
	\label{theorem:Val_mulitplicity_free}
	$\Val^\pm_k(\R^n)$ is a multiplicity-free representation of $\SO(n)$.
\end{theorem}

\section{Lipschitz functions on $\unitsphere$}
We consider real-valued \emph{Lipschitz functions} defined on $\unitsphere$, that is, $f\colon \unitsphere\to\R$ such that there exists a constant $L>0$ for which
$$
|f(x)-f(y)|\le L\ |x-y|,\quad\forall\, x,y\in \unitsphere. 
$$
The smallest number $L$ such that the previous condition holds is called the \emph{Lipschitz constant} of $f$, and we will denote it by $\lip(f)$. The space of Lipschitz functions on $\unitsphere$ will be denoted by $\lip(\unitsphere)$.

\medskip 

Given $f\colon \unitsphere\to\R$, we denote by $\tilde f\colon \R^n\to\R$ its $1$-homogeneous extension to $\R^n$. If $f\colon \unitsphere\rightarrow\R$ is differentiable at $x\in \unitsphere$, let $\nabla f(x)\in T_x\unitsphere=x^\perp\subset \R^n$ denote the \emph{spherical gradient} of $f$ in $x$, that is, the unique tangent vector such that $df(x)v=\langle\nabla f(x),v\rangle $ for all $v\in T_x\unitsphere$. Here $T_x\unitsphere$ is the tangent space to $\unitsphere$ at $x$, and $df$ is the differential of $f$. Similarly, if $\tilde{f}$ is differentiable at $x\in \R^n$, the Euclidean gradient $\nabla_e \tilde{f}(x)$ of $\tilde{f}:\R^n\rightarrow\R$ in $x\in \R^n$ verifies the relation $\langle \nabla_e \tilde{f}(x),v\rangle=d\tilde{f}(x)v$, for all $v\in\R^n$. The following result is well known; we include its proof for completeness.

			\begin{lemma}
				\label{lemma:relation-spherical-euclidean-gradient}
				Let $f:\unitsphere\rightarrow\R$ be a Lipschitz function, $\tilde{f}:\R^n\rightarrow \R$ its $1$-homogeneous extension. Then 
				\begin{equation}\label{Euler relation for gradients}
				\nabla f(x)=\nabla_e\tilde{f}(x)-f(x)x\quad \text{for }\hm\text{-a.e. }x\in \unitsphere.
				\end{equation}
			\end{lemma}
			\begin{proof}
			Note that $\tilde{f}$ is differentiable for $\mathcal{H}^{n}$-a.e. $x\in \R^n$ by Rademacher's theorem. If $\tilde{f}$ is differentiable at $x_0\in \R^n\setminus\{0\}$, then it is also differentiable at $\frac{x_0}{|x_0|}$ with $d\tilde{f}(\frac{x_0}{|x_0|})=d\tilde{f}(x_0)$, as
				\begin{eqnarray*}
				0&=&\lim\limits_{v\rightarrow0}\frac{\tilde{f}(x_0+v)-\tilde{f}(x_0)-d\tilde{f}(x_0)v}{|v|}=	\lim\limits_{v\rightarrow0}\frac{|x_0|\left[\tilde{f}\left(\frac{x_0+v}{|x_0|}\right)-\tilde{f}\left(\frac{x_0}{|x_0|}\right)-d\tilde{f}(x_0)\frac{v}{|x_0|}\right]}{|v|}\\
				&=&\lim\limits_{v\rightarrow0}\frac{\tilde{f}\left(\frac{x_0}{|x_0|}+v\right)-\tilde{f}\left(\frac{x_0}{|x_0|}\right)-d\tilde{f}(x_0)v}{|v|}.
				\end{eqnarray*}
				Using polar coordinates and Fubini's theorem, this implies that the set $A:=\{x\in \unitsphere:\tilde{f}\text{ is differentiable at }x\}$ satisfies $\hm(\unitsphere\setminus A)=0$. Let $B$ denote the set of all points of differentiability of $f$ in $\unitsphere$. Then $\hm(\unitsphere\setminus B)=0$ by Rademacher's theorem.

			If $\tilde{f}$ is differentiable at $x\in \R^n$, then 
				\begin{align*}
				\tilde{f}(x)=\frac{d}{dt}\Big|_{t=0}\tilde{f}(x+tx)=d\tilde{f}(x)x=\langle \nabla_e\tilde{f}(x),x\rangle,
				\end{align*}
				as $\tilde{f}$ is $1$-homogeneous. Thus the scalar product of $x$ and the right-hand side of \eqref{Euler relation for gradients} vanishes and we only have to show that
				\begin{align*}
				df(x)v=d\tilde{f}(x) v
				\end{align*}
				for $x\in A\cap B$ and every $v\in \R^n$ with $\langle x,v\rangle =0$. Let $c:(-\varepsilon,\varepsilon)\rightarrow \unitsphere$, $\varepsilon>0$, be a curve of class $C^1$, with $c(0)=x$ and $c'(0)=v$. As $\tilde{f}$ and $f$ are differentiable at $x$, we can apply the chain rule and obtain
				\begin{align*}
					\label{equation:relation_gradients}
				d\tilde{f}(x)v=\frac{d}{dt}\Big|_{t=0}\tilde{f}(c(t))=\frac{d}{dt}\Big|_{t=0}f(c(t))=df(x)v.
				\end{align*}
				\end{proof}		
			We will also need the following fact.
			\begin{lemma}[\cite{ColesantiEtAlContinuousvaluationsspace2021}, Lemma 3.1]
				\label{lemma:gradient-vanishes-on-constant-sets}
				Let $f:\unitsphere\rightarrow\R$ be a Lipschitz function, $c\in\R$ and let
				\begin{align*}
				Z_c=\{x\in \unitsphere:f(x)=c\}.
				\end{align*}
				Then $\nabla f(x)=0$ for $\hm$-a.e. $x\in Z_c$.
			\end{lemma}
			
\subsection{Topology on $\lip(\unitsphere)$}
	\label{section:lip_basics}
	In \cite{ColesantiEtAlclassinvariantvaluations2020,ColesantiEtAlContinuousvaluationsspace2021} the following notion of convergence was considered.
	\begin{definition}
		\label{def:tauContinuous}
		 A sequence $(f_j)_j$ in $\lip(\unitsphere)$ $\tau$-converges to $f\in \lip(\unitsphere)$ if
	\begin{enumerate}
	\item\label{def:tauContinuousUniformConvergence} $(f_j)_j$ converges uniformly to $f$ on $\unitsphere$;
	\item\label{def:tauContinuousConvergenceGradientAE} $\nabla f_j(x)\rightarrow\nabla f(x)$ for $\hm$-a.e. $x\in\unitsphere$;
	\item\label{def:tauContinuousLipschtitzBound} there exists $C>0$ such that $\|\nabla f_j\|_{L^{\infty}(\unitsphere)}\le C$ for all $j\in\N$.
	\end{enumerate}
	If $\mu:\lip(\unitsphere)\rightarrow X$ is a map into a topological space $X$, then we call $\mu$ $\tau$-continuous if $\mu(f_j)\rightarrow \mu(f)$ for every sequence $(f_j)_j$ that $\tau$-converges to $f$.
\end{definition}
This notion of convergence is not topologizable, compare Corollary \ref{cor:nonTopology} below. However, it leads to the following topology on $\lip(\unitsphere)$: first, let us call a subset $O\subset\lip(\unitsphere)$ \emph{$\tau$-sequentially open} if for every sequence $(f_j)_j$ in $\lip(\unitsphere)$ $\tau$-converging to $f\in O$ there exists $N\in\N$ such that $f_j\in O$ for all $j\ge N$. The family of $\tau$-sequentially open subsets forms a topology on $\lip(\unitsphere)$, which is the coarsest topology such that all $\tau$-convergent sequences are convergent (and the subsequent corollary easily follows from this property). 

Throughout the rest of the paper, $\lip(\unitsphere)$ will be considered as a topological space with this topology.

\begin{corollary}
	\label{corollary:lip_sequential_space}
	A map $\mu:\lip(\unitsphere)\rightarrow X$ into a topological space $X$ is continuous if and only if it is $\tau$-continuous. In particular, 
	$\lip(\unitsphere)$ is a sequential space, that is, every sequentially continuous map $\mu:\lip(\unitsphere)\rightarrow X$ into a topological space $X$ is continuous.
\end{corollary}
\begin{proof}
	Let $\mu:\lip(\unitsphere)\rightarrow X$ be $\tau$-continuous. If $O\subset X$ is open, then $\mu^{-1}(O)$ is $\tau$-sequentially open and thus open. Thus $\mu$ is continuous. Conversely, any continuous map is sequentially continuous, and then also $\tau$-continuous, because $\tau$-convergent sequences are convergent. In particular, any sequentially continuous map is continuous.
\end{proof}

We will need the following property of compact subsets.
\begin{proposition}
	\label{proposition_bounds_on_value_and_gradient}
	Let $K\subset\lip(\unitsphere)$ be compact. Then there exists a constant $C>0$ such that
	\begin{enumerate}
		\item $\|f\|_\infty\le C$ for all $f\in K$;
		\item $\lip(f)\le C$ for all $f\in K$.
	\end{enumerate}
\end{proposition}
\begin{proof}
	For the first property, observe that the inclusion $\lip(\unitsphere)\hookrightarrow C(\unitsphere)$ it $\tau$-continuous by definition. Thus it is continuous and the image of $K$ in $C(\unitsphere)$ is compact and in particular bounded.
	
	For the second property, let us consider the map
	\begin{align*}
	S:\lip(\unitsphere)&\rightarrow L^\infty(\unitsphere,T\unitsphere)\\
	f&\mapsto \nabla f,
	\end{align*}
	where we equip $L^\infty(\unitsphere,T\unitsphere)$ with the weak-$*$ topology, that is, we consider $L^\infty(\unitsphere,T\unitsphere)$ as the dual space $L^1(\unitsphere,T\unitsphere)'$ of $L^1(\unitsphere,T\unitsphere)$. Using the Dominated Convergence Theorem and properties \ref{def:tauContinuousConvergenceGradientAE} and \ref{def:tauContinuousLipschtitzBound} of Definition \ref{def:tauContinuous}, we see that $S$ is $\tau$-continuous and thus continuous. 
	
	Assume that the Lipschitz constants of the functions belonging to $K$ are not uniformly bounded. Then we can find a sequence $(f_j)_j$ in $K$ converging to $f\in K$ such that 
$$
\lim_{j\to\infty}\|\nabla f_j\|_{L^\infty(\unitsphere)}=\infty.
$$
However, the sequence $S(f_j)=\nabla f_j$ is convergent in the weak-$*$ topology. As $L^1(\unitsphere, T\unitsphere)$ is a Banach space, the weakly-$*$ convergent sequence 
$$
(S(f_j))_j\subset L^1(\unitsphere, T\unitsphere)'\cong L^\infty(\unitsphere,T\unitsphere),
$$ 
has to be norm bounded due to the principle of uniform boundedness, which is a contradiction.
\end{proof}

Let us add some final remarks on the topology that will lead us to a characterization of all convergent sequences in $\lip(\unitsphere)$. In \cite{ColesantiEtAlContinuousvaluationsspace2021} the following metric on $\lip(\unitsphere)$ was considered:
		\begin{equation*}
			d_\tau(f,h):=\|f-h\|_\infty+\int_{\unitsphere}|\nabla f(x)-\nabla h(x)|d\hm(x).
		\end{equation*}
 For $M>0$, let us denote by $\lip_M(\unitsphere)\subset \lip(\unitsphere)$ the subspace of all $f\in\lip(\unitsphere)$ such that $\|f\|_{\lip}\le M$, where $\|f\|_{\lip}$ is the norm defined by
$$
\|f\|_{\lip}:=\max(\lip(f),\|f\|_\infty),\quad\forall\, f\in\lip(\unitsphere).
$$

\begin{lemma}
	\label{lemma:characterization-tau-metric}
	A map $\mu:\lip(\unitsphere)\rightarrow X$ into a topological space $X$ is $\tau$-continuous (and thus continuous) if and only if its restriction to $\lip_M(\unitsphere)$ is continuous with respect to the topology induced by $d_\tau$, for all $M>0$.
\end{lemma}
\begin{proof}
	In \cite{ColesantiEtAlContinuousvaluationsspace2021}, Lemma 2.2, this is proved for real-valued valuations $\mu:\lip(\unitsphere)\rightarrow \R$, however, the proof extends almost verbatim to the more general setting.
\end{proof}
	Corollary \ref{corollary:lip_sequential_space} thus implies that $\lip(\unitsphere)$ is the inductive limit in the category of topological spaces with respect to the inclusions of the metric spaces $(\lip_M(\unitsphere),d_\tau)\rightarrow \lip(\unitsphere)$. Indeed, by Proposition \ref{proposition_bounds_on_value_and_gradient} a map $\mu:\lip(\unitsphere)\rightarrow X$ into a topological space $X$ is $\tau$-continuous (and thus continuous) if and only if its restriction to $\lip_M(\unitsphere)$ is continuous for all $M>0$. But by Lemma \ref{lemma:characterization-tau-metric} the topology on $\lip_M(\unitsphere)$ is induced by $d_\tau$. It is easy to see that this implies that $\lip(\unitsphere)$ satisfies the universal property of the inductive limit.
	
	As a consequence, we obtain the following description of convergent sequences in $\lip(\unitsphere)$.
	\begin{proposition}
		\label{prop:characterizationSequences}
		Let $(f_j)_j$ be a sequence in $\lip(\unitsphere)$ and $f\in\lip(\unitsphere)$. Then $(f_j)_j$ converges to $f$ if and only if 
		\begin{enumerate}
			\item $f_j\rightarrow f$ uniformly on $\unitsphere$;
			\item $\nabla f_j\rightarrow\nabla f$ in $L^1(\unitsphere,T\unitsphere)$;
			\item there exists $C>0$ such that $\|\nabla f_j\|_{L^{\infty}(\unitsphere)}\le C$ for all $j\in\N$.
		\end{enumerate}
	\end{proposition}
	\begin{proof}
		If $(f_j)_j$ converges to $f$, then $\{f_j:j\in\N\}\cup\{f\}$ is a compact subset, so Proposition \ref{proposition_bounds_on_value_and_gradient} implies that there exists $M>0$ such that $\|f_j\|_\lip\le M$ for all $j\in\N$ and $\|f\|_\lip\le M$. In particular, $\|\nabla f_j\|_{L^\infty(\unitsphere)}\le M$ for all $j\in\N$, which shows the last property. Thus $(f_j)_j$ converges to $f$ in $\lip_M(\unitsphere)$. By Lemma \ref{lemma:characterization-tau-metric}, this implies that $(f_j)_j$ converges to $f$ with respect to the metric $d_\tau$ and from the definition of this metric we see that $f_j\rightarrow f$ uniformly on $\unitsphere$ and $\nabla f_j\rightarrow \nabla f$ in $L^1(\unitsphere,T\unitsphere)$.\\
		
		Vice versa, if $(f_j)_j$ and $f$ satisfy the properties stated in the proposition, then there exists $M>0$ such that $\|f\|_\lip\le M$ and $\|f_j\|_\lip\le M$ for all $j\in\N$, so $(f_j)_j$ is a sequence in $\lip_M(\unitsphere)$. As the topology on $\lip_M(\unitsphere)$ is induced by $d_\tau$, it is thus enough to show that $(f_j)_j$ converges to $f$ with respect to this metric, which follows directly from the first two properties.
	\end{proof}
	\begin{corollary}
		\label{cor:nonTopology}
		There does not exist a topology on $\lip(\unitsphere)$ such that the convergent sequences are precisely the $\tau$-convergent sequences.
	\end{corollary}
	\begin{proof}
		Using Proposition \ref{prop:characterizationSequences}, it is easy to construct a convergent sequence $(f_j)_j$ in $\lip(\unitsphere)$ that is not $\tau$-convergent. We fix such a sequence. If $\tilde{\tau}$ is a topology on $\lip(\unitsphere)$ such that all $\tau$-convergent sequences are convergent, then the identity map $\lip(\unitsphere)\mapsto (\lip(\unitsphere),\tilde{\tau})$ is $\tau$-sequentially continuous and thus continuous by the definition of the topology on $\lip(\unitsphere)$. Consequently, the sequence $(f_j)_j$ chosen before has to converge with respect to the topology $\tilde{\tau}$ as well. In particular, there is no topology such that the convergent sequences are precisely the $\tau$-convergent sequences.
	\end{proof}

\subsection{$\GL(n,\R)$-operation}
	\label{section:lip_GL_operation}
	We will consider the following operation of $\GL(n,\R)$ on $\lip(\unitsphere)$, which corresponds to the natural operation of $\GL(n,\R)$ on the space of $1$-homogeneous functions defined on the dual space of $\R^n$.
	\begin{definition}
		Given $f\in\lip(\unitsphere)$ and $g\in\GL(n,\R)$, we define $g\cdot f\in \lip(\unitsphere)$ by
		\begin{equation}
		\label{equation:operation-GL-on-lip}
		(g\cdot f)(x):=|g^Tx| f\left(\frac{g^Tx}{|g^Tx|}\right)
		\end{equation}
		for $g\in\GL(n,\R)$, $x\in \unitsphere$, where $g^T$ denotes the transpose of $g$.
	\end{definition}
	The motivation for this definition comes from the following compatibility with support functions.
	\begin{lemma}
		\label{lemma:supportFunctionEquivariant}
		The map
		\begin{align*}
			\K(\R^n)&\rightarrow\lip(\unitsphere)\\
			K&\mapsto h_K
		\end{align*}
		is continuous and $\GL(n,\R)$-equivariant, that is, $h_{gK}(x)=|g^Tx| h_K\left(\frac{g^Tx}{|g^Tx|}\right)$ for $g\in\GL(n,\R)$ and $x\in\unitsphere$. 
	\end{lemma}
	\begin{proof}
		The equivariance of this map follows from the fact that support functions are $1$-homogeneous: if $K\in\K(\R^n)$ and $g\in\GL(n,\R)$, then for all $x\in \unitsphere$
		\begin{align*}
			h_{gK}(x)=\sup_{y\in gK}\langle y,x\rangle=\sup_{y\in K}\langle gy,x\rangle=\sup_{y\in K}\langle y,g^Tx\rangle=|g^Tx|\sup_{y\in K}\left\langle y,\frac{g^Tx}{|g^Tx|}\right\rangle=|g^Tx|h_K\left(\frac{g^Tx}{|g^Tx|}\right).
		\end{align*}
		Let us show that this map is continuous. By the proof of Lemma 2.6 in \cite{ColesantiEtAlclassinvariantvaluations2020}, $K_j\rightarrow K$ in $\mathcal{K}(\R^n)$ implies that $(h_{K_j})_j$ is $\tau$-convergent to $h_K$, so this map is sequentially continuous. But $\K(\R^n)$ is a metric space, so this map is continuous.
	\end{proof}
	In contrast to the spherical gradient of $f$, the Euclidean gradient of the $1$-homogeneous extension $\tilde{f}$ of $f$ behaves nicely with respect to the operation of $\GL(n,\R)$ on $\lip(\unitsphere)$. Let us thus consider the $\hm$-a.e. defined function $\bar{\nabla} f\in L^\infty(\unitsphere,\R^n)\subset L^1(\unitsphere,\R^n)$ defined for $f\in\lip(\unitsphere)$ by \begin{align*}
		\bar{\nabla} f(x):=\nabla f(x)+f(x)x\in \R^n,\quad \text{for  }\hm\text{-a.e. }x\in \unitsphere.
	\end{align*} 
	By Lemma \ref{lemma:relation-spherical-euclidean-gradient}, $\bar{\nabla}f(x)=\nabla_e \tilde{f}(x)$ for $\hm$-a.e. $x\in \unitsphere$. If $K\in\mathcal{K}(\R^n)$ is a smooth convex body with positive Gauss curvature, then it is easy to see that $\bar{\nabla}h_K$ is the inverse Gauss map (see e.g. \cite{SchneiderConvexbodiesBrunn2014}, Corollary 1.7.3).
	\begin{lemma} The map
		\label{lemma:continuity-bar-nabla_L1}
		\begin{align*}
		\bar{\nabla} :\lip(\unitsphere)&\rightarrow L^1(\unitsphere,\R^n)\\
		f&\mapsto \bar{\nabla} f 
		\end{align*}
		is continuous and $\GL(n,\R)$-equivariant, that is, 
		$$
		\bar{\nabla}(g\cdot f)(x)=g\left(\bar{\nabla} f\left(\frac{g^Tx}{|g^Tx|}\right)\right)
		$$ 
		for $g\in\GL(n,\R)$, $f\in\lip(\unitsphere)$ and $\hm$-a.e. $x\in\unitsphere$.
	\end{lemma}
	\begin{proof}
		The equivariance of $\bar{\nabla}$ follows from the fact that $\bar{\nabla} f=\nabla_e \tilde{f}$ $\hm$-a.e. on $\unitsphere$, the properties of the Euclidean gradient and the $0$-homogeneity of $\nabla_e \tilde{f}$.\\
		
		To see that $\bar{\nabla}$ is continuous, note that $|\bar{\nabla}f(x)|$ is bounded by $\|f\|_\infty+\lip(f)$ for $\hm$-a.e. $x\in\unitsphere$. If $(f_j)_j$ $\tau$-converges to $f\in\lip(\unitsphere)$, then $(\bar{\nabla}f_j)_j$ is a bounded sequence in $L^\infty(\unitsphere,\R^n)$ that converges $\hm$-a.e., so it converges in $L^1(\unitsphere,\R^n)$ by the Dominated Convergence Theorem. Thus, $\bar{\nabla}$ is $\tau$-continuous and therefore continuous.
	\end{proof}
	
	We are now in the position to show that the operation of $\GL(n,\R)$ on $\lip(\unitsphere)$ defined above is continuous. The proof makes use of the following result on products of sequential spaces.
		\begin{theorem}[\cite{TanakaProductssequentialspaces1976}, Theorem 1.1.]
		\label{thm:productTopologySequential} 
		Let $X$ be a sequential space such that all one point sets are $G_\delta$-sets (that is, the intersection of a countable family of open subsets). If $Y$ is a locally countably compact, first countable space, then $X\times Y$ equipped with the product topology is sequential.
	\end{theorem}
	
	\begin{proposition}
		\label{proposition:continuity_GL(V)_on_lip}
		The map 
		\begin{align*}
		\GL(n,\R)\times\lip(\unitsphere)&\rightarrow \lip(\unitsphere)\\
		(g,f)&\mapsto g\cdot f
		\end{align*}
		is continuous.
	\end{proposition}
	\begin{proof}
		Let us first show that $\GL(n,\R)\times \lip(\unitsphere)$ is a sequential space by applying Theorem \ref{thm:productTopologySequential} to $X=\lip(\unitsphere)$, which is a sequential space by Corollary \ref{corollary:lip_sequential_space},  and $Y=\GL(n,\R)$. First, observe that $\GL(n,\R)$ is metrizable and locally compact, so we only need to show that $\{f_0\}\subset\lip(\unitsphere)$ is a $G_\delta$-set for every $f_0\in\lip(\unitsphere)$. We can write
		\begin{align*}
		\{f_0\}=\bigcap\limits_{j\in\N}\left\{f\in\lip(\unitsphere):\|f-f_0\|_\infty<\frac{1}{j}\right\},
		\end{align*}
		where the sets on the right are open as the inclusion $\lip(\unitsphere)\hookrightarrow C(\unitsphere)$ is continuous. Therefore, we can apply Theorem \ref{thm:productTopologySequential} to deduce that $\GL(n,\R)\times\lip(\unitsphere)$ is a sequential space. It is thus sufficient to show that
		\begin{align*}
		\GL(n,\R)\times\lip(\unitsphere)&\rightarrow \lip(\unitsphere)\\
		(g,f)&\mapsto g\cdot f
		\end{align*}
		is sequentially continuous.\\
		
		Let $(g_j,f_j)\in \GL(n,\R)\times\lip(\unitsphere)$, $ j\in\N $, be a sequence such that $(g_j,f_j)\rightarrow (g_0,f_0)$. Proposition \ref{prop:characterizationSequences} and Lemma \ref{lemma:continuity-bar-nabla_L1} imply that there exists $C>0$ such that 
		\begin{enumerate}
			\item $f_j\rightarrow f$ uniformly,
			\item $\|\bar{\nabla}f_j\|_{L^\infty(\unitsphere)}\le C$ for all $j\in\N$, 
			\item $\bar{\nabla}f_j\rightarrow \bar{\nabla}f$ in $L^1(\unitsphere,\R^n)$.
		\end{enumerate}
		Thus 
		\begin{align*}
			\|\bar{\nabla}(g_j\cdot f_j)\|_{L^\infty(\unitsphere)}=\sup_{x\in \unitsphere}| g_j\bar{\nabla}f_j\left(g_j^{T}x/|g_j^Tx|\right)|\le C\|g_j\|\le \tilde{C}
		\end{align*} for all $j\in\N$, where $\|g\|$ denotes the operator norm of $g\in \GL(n,\R)$ as a linear map on $\R^n$ and $\tilde{C}:=\sup_{j\in\N}C\|g_j\| $ is finite due to the convergence of $(g_j)_j\subset \GL(n,\R)$. In particular, the Lipschitz constants of $g_j\cdot f_j$ are uniformly bounded. Similarly, $\|g_j\cdot f_j\|_\infty$ is uniformly bounded. Thus $g_j\cdot f_j\in \lip_M(\unitsphere)$ for all $j\in\N$ and some $M>0$. We claim that $(g_j\cdot f_j)_j$ converges to $g_0\cdot f_0$ in $\lip(\unitsphere)$. As this sequence is contained in $\lip_M(\unitsphere)$, we only have to show that it converges with respect to the metric $d_\tau$, due to Lemma \ref{lemma:characterization-tau-metric}. Note that
		\begin{align*}
		d_\tau(g_j\cdot f_j,g_0\cdot f_0)=&\|g_j\cdot f_j-g_0\cdot f_0\|_\infty+\int_{\unitsphere}|\nabla(g_j\cdot f_j)-\nabla(g_0\cdot f_0)|d\hm\\
		\le&c_0\|g_j\cdot f_j-g_0\cdot f_0\|_\infty+\int_{\unitsphere}|\bar{\nabla}(g_j\cdot f_j)-\bar{\nabla}(g_0\cdot f_0)|d\hm\\
		=&c_0\|g_j\cdot f_j-g_0\cdot f_0\|_\infty+\|\bar{\nabla}(g_j\cdot f_j)-\bar{\nabla}(g_0\cdot f_0)\|_{L^1(\unitsphere,\R^n)},
		\end{align*}
		where $ c_0=1+\hm(\unitsphere) $. Consider the operations of $\GL(n,\R)$ on $C(\unitsphere)$ and $L^1(\unitsphere,\R^n)$ respectively, given by
		\begin{align*}
		(g\cdot h)(x):=&|g^Tx|h\left(\frac{g^Tx}{|g^Tx|}\right) &&\text{for }g\in\GL(n,\R),~h\in C(\unitsphere),\\
		(g\cdot h)(x):=&g \ h\left(\frac{g^Tx}{|g^Tx|}\right) &&\text{for }g\in\GL(n,\R),~ h\in L^1(\unitsphere,\R^n).
		\end{align*}
		Due to the compactness of $\unitsphere$ the operation of $\GL(n,\R)$ on $C(\unitsphere)$ is continuous with respect to $\|\cdot\|_\infty$, that is, the map
		\begin{align*}
			\GL(n,\R)\times C(\unitsphere)&\rightarrow C(\unitsphere)\\
			(g,h)&\mapsto g\cdot h
		\end{align*}
		is jointly continuous. Let us show that the operation on $L^1(\unitsphere,\R^n)$ is continuous. As the map $(g,x)\mapsto g^Tx/|g^Tx|$ is smooth, the pushforward of $\hm$ on $\unitsphere$ by the map $x\mapsto  g^Tx/|g^Tx|$ is absolutely continuous with respect to $\hm$. In fact, the density depends continuously on both $x\in \unitsphere$ and $g\in \GL(n,\R)$, so it is in particular uniformly bounded by some $C=C(U)>0$ for $(x,g)\in \unitsphere\times U$, where $U\subset \GL(n,\R)$ is an arbitrary open and relatively compact subset. If $g\in U$ and $\varphi_1,\varphi_2\in L^1(\unitsphere,\R^n)$ are given, this implies
		\begin{align*}
		\|g\cdot \varphi_1-g\cdot \varphi_2\|_{L^1(\unitsphere,\R^n)}=&\int_{\unitsphere}\left|g \varphi_1\left(\frac{g^Tx}{|g^Tx|}\right)-g \varphi_2\left(\frac{g^Tx}{|g^Tx|}\right)\right|d\hm(x)\\
		\le& \|g\|\int_{\unitsphere}\left|\varphi_1\left(\frac{g^Tx}{|g^Tx|}\right)- \varphi_2\left(\frac{g^Tx}{|g^Tx|}\right)\right|d\hm(x)\\
		\le& \|g\| C \|\varphi_1-\varphi_2\|_{L^1(\unitsphere,\R^n)}.
		\end{align*}
		In particular, every $g\in\GL(n,\R)$ operates by a continuous linear map on $L^1(\unitsphere,\R^n)$. It is thus sufficient to show that the operation is continuous at the identity $e\in \GL(n,\R)$. Fix a relatively compact neighborhood $U$ of $e$ such that $\|g\cdot \varphi\|_{L^1(\unitsphere,\R^n)}\le \|g\|C\|\varphi\|_{L^1(\unitsphere,\R^n)}$ for all $g\in U$, $\varphi\in L^1(\unitsphere,\R^n)$ and the constant $C=C(U)$ defined above. Given $\varphi_1,\varphi_2\in L^1(\unitsphere,\R^n)$ and $g\in U$, as well as $f\in C(\unitsphere,\R^n)$, we calculate
		\begin{align*}
		\|g\cdot \varphi_1-\varphi_2\|_{L^1(\unitsphere,\R^n)}\le& \|g\cdot\varphi_1-g\cdot\varphi_2\|_{L^1(\unitsphere,\R^n)}+\|g\cdot\varphi_2-g\cdot f\|_{L^1(\unitsphere,\R^n)}\\
		&+\|g\cdot f-f\|_{L^1(\unitsphere,\R^n)}+\|f-\varphi_2\|_{L^1(\unitsphere,\R^n)}\\
		\le& \|g\|C(\|\varphi_1-\varphi_2\|_{L^1(\unitsphere,\R^n)}+\|\varphi_2-f\|_{L^1(\unitsphere,\R^n)})\\
		&+\hm(\unitsphere)\|g\cdot f-f\|_\infty+\|\varphi_2-f\|_{L^1(\unitsphere,\R^n)}\\
		\le& \tilde{C}(\|\varphi_1-\varphi_2\|_{L^1(\unitsphere,\R^n)}+\|\varphi_2-f\|_{L^1(\unitsphere,\R^n)})\\
		&+\hm(\unitsphere)\|g\cdot f-f\|_\infty,
		\end{align*}
		for some $\tilde{C}>0$ independent of $\varphi_1,\varphi_2,f$, as the operator norm of $g\in\GL(n,\R)$ is bounded on $U$.\\
		
		Let $\varepsilon>0$ be given and fix $\varphi_2\in L^1(\unitsphere,\R^n)$. We may choose $f\in C(\unitsphere,\R^n)$ such that $\|\varphi_2-f\|_{L^1(\unitsphere,\R^n)}<\varepsilon/(3\tilde{C})$. If $\varphi_1\in L^1(\unitsphere,\R^n)$ is a function with $\|\varphi_1-\varphi_2\|_{L^1(\unitsphere,\R^n)}<\varepsilon/(3\tilde{C})$, then
		\begin{align*}
		\|g\cdot \varphi_1-\varphi_2\|_{L^1(\unitsphere,\R^n)}\le \frac{2\varepsilon}{3}+\hm(\unitsphere)\|g\cdot f-f\|_\infty,\quad \forall g\in U.
		\end{align*}
		Due to the compactness of $\unitsphere$, the restriction of the operation to $C(\unitsphere,\R^n)\subset L^1(\unitsphere,\R^n)$ is continuous with respect to $\|\cdot\|_\infty$, so we may choose an open neighborhood $U'\subset U$ of the identity such that $\hm(\unitsphere)\|g\cdot f-f\|_\infty<\varepsilon/3$ for all $g\in U'$. Thus $\|\varphi_1-\varphi_2\|_{L^1(\unitsphere,\R^n)}<\varepsilon/(3\tilde{C})$ and $g\in U'$ imply
		\begin{align*}
		\|g\varphi_1-\varphi_2\|_{L^1(\unitsphere,\R^n)}<\varepsilon.
		\end{align*}
		Therefore the operation on $L^1(\unitsphere,\R^n)$ is continuous.\\
		
		As $\bar{\nabla}:\lip(\unitsphere)\rightarrow L^1(\unitsphere,\R^n)$ is continuous by Lemma \ref{lemma:continuity-bar-nabla_L1}, the continuity of the operation of $\GL(n,\R)$ on $L^1(\unitsphere,\R^n)$ and our previous calculation imply
		\begin{align*}
		\limsup\limits_{j\rightarrow\infty}d_\tau(g_j\cdot f_j,g_0\cdot f_0)\le& \limsup\limits_{j\rightarrow\infty}\left(c_0\|g_j\cdot f_j-g_0\cdot f_0\|_\infty+\|\bar{\nabla}(g_j\cdot f_j)-\bar{\nabla}(g_0\cdot f_0)\|_{L^1(\unitsphere,\R^n)}\right)\\
		=& \lim\limits_{j\rightarrow\infty}\|g_j\cdot \bar{\nabla}f_j-g_0\cdot \bar{\nabla}f\|_{L^1(\unitsphere,\R^n)}\\
		=&0.
		\end{align*}
		Thus $(g_j\cdot f_j)_j$ converges to $g_0\cdot f_0$ in $\lip_M(\unitsphere)$, which shows that the map is sequentially continuous. This finishes the proof.
	\end{proof}
						
\section{Valuations on Lipschitz functions on $\unitsphere$}
	\subsection{Topology on $\Val(\lip(\unitsphere))$}
			Let us equip $\Val(\lip(\unitsphere))$ with the topology of uniform convergence on compact subsets in $\lip(\unitsphere)$, that is, the topology induced by the family of semi-norms
			\begin{align*}
				\|\mu\|_K:=\sup\limits_{f\in K}|\mu(f)|,
			\end{align*}
			for $K\subset\lip(\unitsphere)$ compact. Note that, as $\mu$ is complex-valued, here $|\cdot|$ denotes the norm in $\C$.
			
		\begin{proposition}
			\label{prop:ValLipComplete}
			$\Val(\lip(\unitsphere))$ is a complete locally convex space.
		\end{proposition}
		\begin{proof}
			Let $\Lambda$ be a directed set. Given a Cauchy net $(\mu_\lambda)_{\lambda\in\Lambda}$ in $\Val(\lip(\unitsphere))$, $(\mu_\lambda(f))_{\lambda\in\Lambda}$ is a Cauchy net in $\C$ for all $f\in \lip(\unitsphere)$ and thus converges in $\C$. Let us denote this pointwise limit by $\mu(f)$. This defines a map $\mu\colon\lip(\unitsphere)\to\C$. To see that $\mu$ is continuous, note that for $f\in\lip(\unitsphere)$ and a sequence $(f_j)_j$ converging to $f$, the set
			\begin{align*}
				K:=\{f_j:j\in\N\}\cup\{f\}
			\end{align*}
			is compact in $\lip(\unitsphere)$. Therefore, for every $j\in\N$,
			\begin{align*}
				|\mu(f_j)-\mu(f)|\le& |\mu(f_j)-\mu_\lambda(f_j)|+|\mu_\lambda(f_j)-\mu_\lambda(f)|+|\mu_\lambda(f)-\mu(f)|\\
				\le &\limsup\limits_{\eta\in\Lambda}|\mu_\lambda(f_j)-\mu_\eta(f_j)|+|\mu_\lambda(f_j)-\mu_\lambda(f)|+|\mu_\lambda(f)-\mu(f)|\\
				\le&\limsup\limits_{\eta\in\Lambda}\|\mu_\lambda-\mu_\eta\|_K+|\mu_\lambda(f_j)-\mu_\lambda(f)|+|\mu_\lambda(f)-\mu(f)|,
			\end{align*}
			for any $\lambda\in\Lambda$. Given $\varepsilon>0$, choosing $\lambda$ appropriately, the first and last term can be made smaller than $\frac{\varepsilon}{3}$. By continuity, $\mu_\lambda(f_j)$ converges to $\mu_\lambda(f)$ for $j\rightarrow \infty$, so this term can be made smaller than $\frac{\varepsilon}{3}$ for $j\ge N$, for some $N\in\N$. Thus $|\mu(f_j)-\mu(f)|<\varepsilon$ for all $j\ge N$, that is, $(\mu(f_j))_j$ converges to $\mu(f)$. In particular, $\mu$ is sequentially continuous and thus continuous by Corollary \ref{corollary:lip_sequential_space}. As $(\mu_{\lambda})_{\lambda\in\Lambda}$ converges pointwise to $\mu$, it is straightforward to check that $\mu$ is a valuation and that $\mu$ is dually translation invariant. Thus $\mu\in\Val(\lip(\unitsphere))$.\\
			To see that $(\mu_\lambda)_{\lambda\in\Lambda}$ converges to $\mu$, let $K\subset \lip(\unitsphere)$ be a compact subset. Then
			\begin{align*}
				\|\mu-\mu_\lambda\|_K\le \limsup\limits_\eta \|\mu_\eta-\mu_\lambda\|_K.
			\end{align*}
			As $(\mu_\lambda)_\lambda$ is a Cauchy net, there exists $\lambda_0$ such that the expression on the right is smaller than $\varepsilon$ for all $\lambda>\lambda_0$. Thus the net $(\mu_\lambda)_{\lambda\in\Lambda}$ converges to $\mu$ in $\Val(\lip(\unitsphere))$.
		\end{proof}
	$\Val(\lip(\unitsphere))$ becomes a representation of $\GL(n,\R)$ by setting
	\begin{align*}
	[\pi(g)\mu ](f):=\mu(g^{-1}\cdot f).
	\end{align*}
	\begin{corollary}
		The map 
		\begin{align*}
		\GL(n,\R)\times\Val(\lip(\unitsphere))&\rightarrow \Val(\lip(\unitsphere))\\
		(g,\mu)&\mapsto \pi(g)\mu
		\end{align*}
		is continuous. In particular, $\Val(\lip(\unitsphere))$ is a continuous representation of $\GL(n,\R)$. 
	\end{corollary}
	
	\begin{proof}
		Let $K\subset\lip(\unitsphere)$ be a compact subset. Fix $g\in\GL(n,\R)$, $\mu\in\Val(\lip(\unitsphere))$ and let $h\in \GL(n,\R)$, $\nu\in\Val(\lip(\unitsphere))$. Then
		\begin{align*}
		\|\pi(g)\mu-\pi(h)\nu\|_K\le& \|\pi(g)\mu-\pi(h)\mu\|_K+\|\pi(h)\mu-\pi(h)\nu\|_K\\
		=&\|\pi(g)\mu-\pi(h)\mu\|_K+\sup\limits_{f\in K}|\mu(h^{-1}\cdot f)-\nu(h^{-1}\cdot f)|\\
		=&\|\pi(g)\mu-\pi(h)\mu\|_K+\|\mu-\nu\|_{h^{-1}(K)}.
		\end{align*}
		
		Let $\varepsilon>0$ be given.
		As the operation of $\GL(n,\R)$ on $\lip(\unitsphere)$ is continuous, due to Proposition \ref{proposition:continuity_GL(V)_on_lip}, the map
		\begin{align*}
		\GL(n,\R)\times K&\rightarrow \C\\
		(h,f)&\mapsto \mu(h^{-1}\cdot f)
		\end{align*}
		is continuous. In particular, it is uniformly continuous on compact subsets, so we can find a relatively compact open neighborhood $U$ of $g$ such that
		\begin{align*}
		|\mu(h^{-1}\cdot f)-\mu(g^{-1}\cdot f)|<\frac{\varepsilon}{2},\quad \forall h\in U, \forall f\in K.
		\end{align*}
		Next observe that the image $\tilde{K}$ of the compact set $U\times K$ under the map $(h,f)\mapsto h^{-1}\cdot f$ is compact in $\lip(\unitsphere)$, as the operation is continuous. In particular, $h^{-1}(K)\subset \tilde{K}$ for all $h\in U$. If $h\in U$ and $\nu\in \Val(\lip(\unitsphere))$ is such that $\|\mu-\nu\|_{\tilde{K}}<\frac{\varepsilon}{2}$, then
		\begin{align*}
		\|\pi(g)\mu-\pi(h)\nu\|_K\le&\|\pi(g)\mu-\pi(h)\mu\|_K+\|\mu-\nu\|_{h^{-1}(K)}\\
		\le &\frac{\varepsilon}{2}+\|\mu-\nu\|_{\tilde{K}}< \varepsilon.
		\end{align*}
		Thus the operation is continuous.
	\end{proof}
	
	\begin{remark}
		It was shown in \textnormal{\cite{ColesantiEtAlContinuousvaluationsspace2021}} that every $\tau$-continuous valuation on $\lip(\unitsphere)$ is bounded on the sets $\lip_M(\unitsphere)$, $M>0$. In particular, one may equip $\Val(\lip(\unitsphere))$ with
		\begin{align*}
		\|\mu\|:=\sup\limits_{f\in\lip_1(\unitsphere)}|\mu(f)|,\quad \text{for }\mu\in\Val(\lip(\unitsphere)),
		\end{align*}
		which is a norm due to the homogeneous decomposition. It is easy to see that $\Val(\lip(\unitsphere))$ is complete with respect to $\|\cdot\|$, however we do not know whether the operation of $\GL(n,\R)$ is continuous with respect to this norm. We are thus forced to work with the topology of uniform convergence on compact subsets, which is finer than the topology induced by this norm (note that every compact set is contained in $\lip_M(\unitsphere)$ for some $M>0$, by Proposition \ref{proposition_bounds_on_value_and_gradient}).
	\end{remark}

\subsection{Relation to translation invariant valuations on convex bodies}
Recall that we have the map
\begin{align*}
	P:\K(\R^n)&\rightarrow\lip(\unitsphere)\\
	K&\mapsto h_K,
\end{align*}
which is continuous and $\GL(n,\R)$-equivariant. If $\mu\in\Val(\lip(\unitsphere))$, then we may define $T(\mu):\K(\R^n)\rightarrow\C$ by
\begin{align*}
	T(\mu)[K]:=\mu(h_K),\quad K\in\K(\R^n).
\end{align*}
\begin{proposition}
	The map $T:\Val_k(\lip(\unitsphere))\rightarrow\Val_k(\R^n)$ is well-defined, continuous, $\GL(n,\R)$-equivariant and injective.
\end{proposition}
\begin{proof}
	The well-definedness of this map follows from \cite{ColesantiEtAlclassinvariantvaluations2020}, Lemma 2.7. The injectivity was shown in \cite{ColesantiEtAlclassinvariantvaluations2020}, Lemma 3.1. To see that $T$ is $\GL(n,\R)$-equivariant, note that $h_{g^{-1}K}=g^{-1}\cdot h_K$ for $g\in\GL(n,\R)$ and $K\in\K(\R^n)$  by Lemma \ref{lemma:supportFunctionEquivariant}, so for $\mu\in\Val(\lip(\unitsphere))$
	\begin{align*}
		T(\mu)[g^{-1}K]=\mu(h_{g^{-1}K})=\mu(g^{-1}\cdot h_K)=T(\pi(g)\mu)[K].
	\end{align*}
	Thus $T$ is $\GL(n,\R)$-equivariant.\\
	
	It remains to see that $T$ is continuous. Let $B_1(0)\subset\R^n$ denote the unit ball. Then $A:=\{K\in\K(\R^n):K\subset B_1(0)\}$ is a compact subset of $\mathcal{K}(\R^n)$ and the topology on $\Val(\R^n)$ is induced by the norm $\|\nu\|:=\sup_{K\in A}|\nu(K)|$. As $P$ is continuous, $P(A)\subset\lip(\unitsphere)$ is compact, so
	\begin{align*}
		\|T(\mu)\|=\sup_{K\in A}|T(\mu)[K]|=\sup_{K\in A}|\mu(h_K)|=\sup_{f\in P(A)}|\mu(f)|=\|\mu\|_{P(A)}.
	\end{align*}
	Thus $T$ is continuous.
\end{proof}

\section{Proof of Theorem \ref{maintheorem:decomposition}}
	\subsection{General considerations and proof for even valuations}
	The key ingredient for our proof is the following observation.
\begin{proposition}
	Let $F$ be a continuous representation of $\GL(n,\R)$ on a complete locally convex vector space $F$ and assume that there is a continuous, $\GL(n,\R)$-equivariant linear map 
	\begin{align*}
	S:F\rightarrow \Val^\pm_k(\R^n)
	\end{align*}
	for some $0\le k\le n$. If $\mathrm{Im}(S)\ne 0$, then $S$ maps the space $F^{\SOnFin}$ of $\SO(n)$-finite vectors onto $\Val_k^\pm(\R^n)^{\SOnFin}$.
\end{proposition}
\begin{proof}

	Obviously, $S$ restricts to a map between the corresponding spaces of $\SO(n)$-finite vectors. As both spaces are complete, the dense subspaces of $\SO(n)$-finite vectors decompose into direct sums of the respective isotypical components (see \cite{SepanskiCompactLiegroups2007}, Theorem 3.46), and $S$ descends to a map between these components by Schur's lemma. Assume that $S$ is not surjective. As $\Val_k^\pm(\R^n)$ contains each irreducible representation of $\SO(n)$ at most once by Theorem \ref{theorem:Val_mulitplicity_free}, there exists a non-trivial isotypical component $E:=\Val_k^\pm(\R^n)[\Gamma]$ for some representation $\Gamma$ of $\SO(n)$ that is not contained in the image of $S$.  We define an $\SO(n)$-equivariant continuous projection $P$ onto this space as follows: take any continuous projection $\tilde{P}:\Val_k^\pm(\R^n)\rightarrow E$ (note that $E$ is finite-dimensional, so such a projection exists) and define
	\begin{align*}
	P:\Val_k^\pm(\R^n)&\rightarrow E\\
	\mu&\mapsto P\mu:=\int_{\SO(n)}\pi(g^{-1}) \tilde{P}\pi(g)\mu dg.
	\end{align*}
	Recall that the topology of $\Val(\R^n)$ is induced by the norm $\|\mu\|=\sup_{K\subset B_1(0)}|\mu(K)|$, so $\SO(n)$ operates on $\Val_k^{\pm}(\R^n)$ by isometries. In addition, the integrand is a continuous $E$-valued function on $\SO(n)$ and thus integrable. In particular, $P$ is well-defined. Moreover	\begin{align*}
	\|P\mu\|\le \int_{\SO(n)}\|\pi(g^{-1}) \tilde{P}\pi(g)\mu\| dg\le \|\tilde{P}\|\cdot \|\mu\|,
	\end{align*}
	where $\|\tilde{P}\|$ denotes the operator norm of $\tilde{P}$. Thus $P$ is continuous. Obviously, $P$ is a $\SO(n)$-equivariant projection onto $E$.
	
	Now note that $S(F^{\SOnFin})\subset \ker P$ by Schur's lemma, as $P\circ S$ is $\SO(n)$-equivariant and all subrepresentations contained in the image of $S$ are inequivalent to $E$ by assumption. As the kernel of $P\circ S$ is closed and $F^{\SOnFin}$ is dense in $F$, $\mathrm{Im}(S)\subset \ker P$. But if $\mathrm{Im}(S)\ne 0$, then $\mathrm{Im}(S)$ is dense in $\Val_k^\pm(\R^n)$ by Alesker's Irreducibility Theorem, so we obtain the contradiction $\Val_k^\pm(\R^n)\subset\ker P$. Thus $S$ has to be surjective on the level of $\SO(n)$-finite vectors.
\end{proof}
If $ T $ is the map defined in the previous section,  this proposition yields the following result.
\begin{corollary}
	\label{corollary_T_isomorphism_on_Finite_vectors}
	If $\Val^\pm_k(\lip(\unitsphere))\ne 0$, then
	\begin{align*} 
	T:\Val^\pm_k(\lip(\unitsphere))^{\SOnFin}\rightarrow\Val_k^\pm(\R^n)^{\SOnFin}
	\end{align*}
	is an isomorphism. 
\end{corollary}
\begin{proof}
	From the previous proposition we deduce that the restriction of $T$ to $\SO(n)$-finite vectors is surjective unless the image of $T$ is zero. Obviously, the restriction of $T$  is also injective. 
\end{proof}
This allows us to prove our main result for even valuations.
\begin{proof}[Proof of Theorem \ref{maintheorem:decomposition} for even valuations]
	By Theorem \ref{theorem:HadwigerLipschitz}, there does not exist any non-trivial rotation invariant valuation in $\Val_k^+(\lip(\unitsphere))$ for $k\ge 3$; that is, the isotypical component of the trivial representation of $\SO(n)$ has multiplicity $0$. As the $k$-th intrinsic volume defines a $\SO(n)$-invariant valuation in $\Val_k^+(\R^n)$, the map $T:\Val^+_k(\unitsphere)^{\SOnFin}\rightarrow \Val_k^+(\R^n)^{\SOnFin}$ is not an isomorphism for $k\ge 3$. Thus $\Val_k^+(\unitsphere)=0$ for $k\ge 3$ by Corollary \ref{corollary_T_isomorphism_on_Finite_vectors}.
\end{proof}

\subsection{Proof for odd valuations}
\label{section:proofOddCase}
The arguments of the previous section also apply to odd valuations. We thus have to show that there exists an odd valuation in $\Val_k^-(\R^n)$ that does not extend to $\lip(\unitsphere)$ and that is $\SO(n)$-finite. We will see that it is enough to consider valuations of degree $k=n-1$.\\

Consider the valuation $\mu_{n-1}\in\Val_{n-1}(\R^n)$ given by
	\begin{align*}
		\mu_{n-1}(K)=\int_{\unitsphere} p(x) dS_{n-1}(K,x),
	\end{align*}
	where $p\colon \unitsphere\to\R$ is defined by
$$
p(x)=p(x_1,\dots,x_n)=x_1^3
$$ 
and $S_{n-1}(K,\cdot)$ denotes the area measure of $K\in\mathcal{K}(\R^n)$ (see \cite{SchneiderConvexbodiesBrunn2014}, Chapter 4). As $p$ is a polynomial, this valuation is contained in a finite-dimensional and $\SO(n)$-invariant subspace of $\Val_{k}(\R^n)$. The calculation below will also show that $\mu_{n-1}$ is non-trivial. Obviously $\mu_{n-1}$ is an odd valuation because $p$ is odd. The key result of this section is the following proposition.

	\begin{proposition}
		\label{proposition_missing_valuation_odd_case}
		Let $n\ge 4$. There exists no valuation $\nu\in\Val_{n-1}(\lip(\unitsphere))$ such that $\mu_{n-1}(K)=\nu(h_K)$ for all $K\in\mathcal{K}(\R^n)$.
	\end{proposition}


We start with a geometric construction needed for the proof and some related results. Let $\xi\in \unitsphere$; we set
$$
D_\xi=\{\eta\in\R^n\colon \eta\in\xi^\perp,\ |\eta|\le1\},
$$
and, for $\lambda\ge 1$, 
$$
C_{\xi,\lambda}=\mbox{conv}\left((\lambda D_\xi-\xi)\cup\{0\}\right),
$$
where $\mbox{conv}(A)$ denotes the convex hull of a set $A\subset\R^n$. $C_{\xi,\lambda}$ is the cone with apex at the origin and base $\lambda D_\xi-\xi$. It will be important to evaluate
$$
\int_{\unitsphere}p(x)dS_{n-1}(C_{\xi,\lambda},x)
$$
when $\xi$ is close to $e_1=(1,0,\dots,0)$. As a first step in this direction, we deduce an explicit expression for the area measure of $C_{\xi,\lambda}$.

\begin{lemma}\label{area measure of the cone} With the previous notations, let
$$
\Sigma_{\xi,\lambda}=\left\{\frac1{\sqrt{\lambda^2+1}}(\lambda\xi+\eta)\colon \eta\in\xi^\perp,\, |\eta|=1\right\}\subset\unitsphere.
$$
Then 
$$
S_{n-1}(C_{\xi,\lambda},\cdot)=\lambda^{n-1}\omega_{n-1}\delta_{-\xi}(\cdot)+\frac{(\lambda^2+1)^{\frac{n-1}{2}}}{\lambda}\,({\mathcal H}^{n-2}\ \llcorner\ \Sigma_{\xi,\lambda})(\cdot).
$$
Here $\omega_{n-1}$ is the Lebesgue measure of the $(n-1)$-dimensional unit ball, $\delta_{\xi}$ is the Dirac point mass measure concentrated at $-\xi$, and $({\mathcal H}^{n-2}\ \llcorner\ \Sigma_{\xi,\lambda})$ denotes the restriction of ${\mathcal H}^{n-2}$ to $\Sigma_{\xi,\lambda}$. 
\end{lemma}

\begin{proof} We say that a boundary point $x$ of $C_{\xi,\lambda}$ is \emph{regular} if $\partial C_{\xi,\lambda}$ admits a unique supporting hyperplane at $x$; in this case we denote by $N(x)$ the corresponding outer unit normal (that is, $N$ is the Gauss map of $C_{\xi,\lambda}$). Given $y\in \unitsphere$, we have
$$
N^{-1}(y)=\{x\in\partial C_{\xi,\lambda}\colon\mbox{$x$ is regular and $N(x)=y$}\}.
$$ 
The boundary of $C_{\xi,\lambda}$ contains the $(n-1)$-dimensional facet $\lambda D_\xi-\xi$, with outer unit normal $-\xi$. Hence (see for instance Theorem 4.5.6 in \cite{SchneiderConvexbodiesBrunn2014}) the area measure of $C_{\xi,\lambda}$ has a Dirac point mass at $-\xi$, and more precisely,
$$
S_{n-1}(C_{\xi,\lambda},\{-\xi\})=\omega_{n-1}\lambda^{n-1}.
$$
We now evaluate $S_{n-1}(C_{\xi,\lambda},\cdot)$ on $\unitsphere\setminus\{-\xi\}$. Let
$$
L=\partial C_{\xi,\lambda}\setminus(\lambda D_\xi-\xi)=\{x\in\R^n\colon x=s(-\xi+\lambda\eta),\ \eta\in\xi^\perp,\ |\eta|=1,\, s\in[0,1)\}.
$$
If $\eta\in\xi^\perp$, $|\eta|=1$, then each point of the form
\begin{equation}\label{segment}
x=s(-\xi+\lambda\eta), \quad s\in(0,1),
\end{equation}
is regular, and 
$$
N(x)=\frac1{\sqrt{\lambda^2+1}}(\lambda\xi+\eta).
$$
Therefore
$$
\Sigma_{\xi,\lambda}:=\left\{\frac1{\sqrt{1+\lambda^2}}(\lambda\xi+\eta)\colon \eta\in\xi^\perp,\, |\eta|=1\right\}
$$
is the image, through the Gauss map $N$, of the regular points of $L$. According to Theorem 4.5.3 of \cite{SchneiderConvexbodiesBrunn2014}, the restriction of $S_{n-1}(C_{\xi,\lambda},\cdot)$ to $\unitsphere\setminus\{-\xi\}$ is supported on $\Sigma_{\xi,\lambda}$.

For the rest of the proof, we assume that $\xi=e_1=(1,0,\dots,0)$. By the rigid motion covariance of area measures (see \cite[\S 4.2]{SchneiderConvexbodiesBrunn2014}), this is not restrictive. We can write $L$ as
$$
L=\left\{\left(-\frac{|x'|}{\lambda},x'\right)\colon x'\in\R^{n-1},\ |x'|<1\right\}
$$
(here $|x'|$ is obviously the norm of $x'$ in $\R^{n-1}$). We have
$$
N\left(-\frac{|x'|}{\lambda},x'\right)=\frac{1}{\sqrt{1+\lambda^2}}\left(\lambda,\frac{x'}{|x'|}\right),
$$
whenever $x'\ne0$. Moreover
$$
\Sigma_{\xi,\lambda}=\left\{\frac{1}{\sqrt{1+\lambda^2}}(\lambda,\eta')\colon \eta'\in{\mathrm{S}^{n-2}}\right\}.
$$	 
Let $\omega$ be a Borel subset of $\Sigma_{\xi,\lambda}$ and let
$$
\bar\omega=\left\{\eta'\in{\mathrm{S}^{n-2}}\colon\frac{1}{\sqrt{1+\lambda^2}}(\lambda,\eta')\in\omega\right\}\subset{\mathrm S}^{n-2}.
$$
Then
$$
N^{-1}(\omega)=\left\{\left(-\frac s\lambda,s\eta'\right)\colon \eta'\in\bar\omega, s\in(0,1)\right\}.
$$
The $(n-1)$-dimensional Hausdorff measure of $N^{-1}(\omega)$ can be explicitly computed (for instance, by expressing $N^{-1}(\omega)$ as the graph of a function of $(n-1)$ variables):
$$
{\mathcal H}^{n-1}(N^{-1}(\omega))=\frac{\sqrt{1+\lambda^2}}{\lambda}{\mathcal H}^{n-2}(\bar\omega).
$$
On the other hand, as $\bar\omega$ is obtained from $\omega$ by a translation and a dilation by the factor $\sqrt{1+\lambda^2}$, we have
$$
{\mathcal H}^{n-2}(\bar\omega)=(1+\lambda^2)^{\frac{n-2}2}{\mathcal H}^{n-2}(\omega).
$$
Therefore, using \cite{SchneiderConvexbodiesBrunn2014}, Theorem 4.2.3,
$$
S_{n-1}(C_{\xi,\lambda},\omega)={\mathcal H}^{n-1}(N^{-1}(\omega))=\frac{(1+\lambda^2)^{\frac{n-1}2}}{\lambda}\mathcal{H}^{n-2}(\omega).
$$
\end{proof}

We have the following estimate.

\begin{lemma}\label{lemma estimate} There exists $\delta\in\left(\frac{1}{2},1\right]$ such that for every $\lambda\ge2$ and for every $\xi=(\xi_1,\dots,\xi_n)\in \unitsphere$ which verifies $\xi_1\ge\delta$, we have
$$
\int_{\unitsphere}p(x)dS_{n-1}(C_{\xi,\lambda},x)\le-\frac{\omega_{n-1}}2\ \lambda^{n-3}.
$$
\end{lemma}

\begin{proof} 

Let $ \xi $ be such that $ \xi_1\geq\frac{1}{2} $. By Lemma \ref{area measure of the cone}
\begin{equation}\label{sum}
\int_{\unitsphere}p(x) dS_{n-1}(C_{\xi,\lambda},x)=\omega_{n-1}\lambda^{n-1}p(-\xi)+\frac{(\lambda^2+1)^{\frac{n-1}{2}}}{\lambda}\int_{\Sigma_{\xi,\lambda}}p(x)\, d{\mathcal H}^{n-2}(x)=I_1+I_2.
\end{equation}
We focus on $I_2$. In what follows, we will use the notations introduced in the proof of Lemma \ref{area measure of the cone}.  Let 
$$
\Sigma_{\xi,\lambda}\ni x=(x_1,\dots,x_n)=\frac{1}{\sqrt{\lambda^2+1}}(\lambda\xi+\eta),\quad \eta\in\xi^\perp,\, |\eta|=1. 
$$
Then
$$
x_1=\langle x,e_1\rangle=\frac{\lambda\xi_1}{\sqrt{\lambda^2+1}}+\frac{\eta_1}{\sqrt{\lambda^2+1}}\ge
\frac{\lambda/2-1}{\sqrt{\lambda^2+1}}\ge0
$$ 
by the assumptions on $\xi_1$ and $\lambda$. By the definition of $p$, we have: 
$$
p(x)=\frac{1}{(\lambda^2+1)^{3/2}}\left[\xi_1^3\lambda^3+3\lambda^2\xi_1^2(\eta,e_1)+3\lambda \xi_1(\eta,e_1)^2+(\eta,e_1)^3\right].
$$
Note that
$$
\eta=\eta(x)=\sqrt{\lambda^2+1}\ x-\lambda\xi.
$$
By the change of variable $\{\xi^\perp\cap \unitsphere\}\ni\zeta\mapsto\frac1{\sqrt{1+\lambda^2}}(\lambda\xi+\zeta)\in\Sigma_{\xi,\lambda}$,
\begin{eqnarray*}
\int_{\Sigma_{\xi,\lambda}}\langle \eta(x),e_1\rangle\, d{\mathcal H}^{n-2}(x)=
(\lambda^2+1)^{(n-2)/2}\int_{\{\xi^\perp\cap \unitsphere\}}\langle\zeta,e_1\rangle d{\mathcal H}^{n-2}(\zeta).
\end{eqnarray*}
The last integral vanishes, as the integrand is an odd function; thus
$$
\int_{\Sigma_{\xi,\lambda}}\langle \eta(x),e_1\rangle\, d S_{n-1}(C_{\xi,\lambda},x)=0.
$$
Similarly
$$
\int_{\Sigma_{\xi,\lambda}}\langle \eta,e_1\rangle^3\, dS_{n-1}(C_{\xi,\lambda},x)=0.
$$
Hence
$$
I_2=\frac{1}{(\lambda^2+1)^{3/2}}\int_{\Sigma_{\xi,\lambda}}\left(\xi_1^3\lambda^3+3\lambda \xi_1\langle \eta,e_1\rangle^2\right)dS_{n-1}(C_{\xi,\lambda},x).
$$
Note that, since $ \eta\in\xi^\perp $,
$$
|\langle \eta,e_1\rangle|=|\langle \eta,\xi\rangle+\langle\eta,e_1-\xi\rangle|\le|e_1-\xi|\quad\mbox{and}\quad \xi_1\le 1.
$$
Therefore
\begin{eqnarray*}
I_2&=&\frac{1}{(\lambda^2+1)^{3/2}}\int_{\Sigma_{\xi,\lambda}}\left(\xi_1^3\lambda^3+3\lambda \xi_1\langle\eta,e_1\rangle^2\right)dS_{n-1}(C_{\xi,\lambda},x)\\
&\le&\left(\frac{\lambda^3\xi_1^3+3\lambda|\xi-e_1|^2}{(\lambda^2+1)^{3/2}}\right)\int_{\Sigma_{\xi,\lambda}}dS_{n-1}(C_{\xi,\lambda},x)\\
&=&\left(\frac{\lambda^3\xi_1^3+3\lambda|\xi-e_1|^2}{(\lambda^2+1)^{3/2}}\right){\mathcal H}^{n-1}(N^{-1}(\Sigma_{\xi,\lambda}))\\
&=&\left(\frac{\lambda^3\xi_1^3+3\lambda|\xi-e_1|^2}{(\lambda^2+1)^{3/2}}\right){\mathcal H}^{n-1}(L)\\
&=&\left(\frac{\lambda^3\xi_1^3+3\lambda|\xi-e_1|^2}{(\lambda^2+1)^{3/2}}\right)\omega_{n-1}\lambda^{n-2}\sqrt{\lambda^2+1}\\
&=&\frac{\omega_{n-1}\lambda^{n-1}}{\lambda^2+1}\left(\lambda^2\xi_1^3+3|\xi-e_1|^2\right).
\end{eqnarray*}
As 
$$
\lim_{\xi\to e_1}(-\xi_1^3+3|\xi-e_1|^2)
=-1,
$$
from the previous relations and \eqref{sum} we deduce that
\begin{eqnarray*}
\int_{\unitsphere}p(x) dS_{n-1}(C_{\xi,\lambda},x)\le\frac{\omega_{n-1}\lambda^{n-1}}{\lambda^2+1}
(-\xi_1^3+3|\xi-e_1|^2)\le\frac{\omega_{n-1}}{2}\lambda^{n-3}(-\xi_1^3+3|\xi-e_1|^2)
\end{eqnarray*}
for $ \xi $ sufficiently close to $ e_1 $, as $\lambda\ge1$. The conclusion follows.
\end{proof}
With these results we are able to prove Proposition \ref{proposition_missing_valuation_odd_case}.
\begin{proof}[Proof of Proposition \ref{proposition_missing_valuation_odd_case}] This proof exploits a similar construction to the one used in \cite[\S 5]{ColesantiEtAlclassinvariantvaluations2020}.

\medskip

\noindent{\bfseries Step 1.} We argue by contradiction: assume that there exists $\nu\in\Val_{n-1}(\lip(\unitsphere))$ such that
$$
\nu(h_K)=\mu_{n-1}(K)=\int_{\unitsphere}p(x) dS_{n-1}(K,x),
$$
for every convex body $K\in\mathcal{K}(\R^n)$. 

\medskip

\noindent{\bfseries Step 2.} We refer to the definitions and the notations introduced so far in this section; in particular, given $\xi\in\unitsphere$ and $\lambda\ge1$ we will be considering the sets $D_\xi$ and $C_{\lambda,\xi}$. The support function of $D_\xi$ is given by
$$
h_{D_\xi}(x)=|x-\langle x,\xi\rangle\xi|.
$$
The support function of $\lambda D_\xi-\xi$, denoted by $\varphi_{\xi,\lambda}$, is
$$
\varphi_{\xi,\lambda}(x)=\lambda h_{D_\xi}(x)-\langle x,\xi\rangle.
$$

We denote by $\mathbb{O}$ the function identically equal to $0$ on $\unitsphere$; let $\psi_{\xi,\lambda}\colon \unitsphere\to\R$ be defined by
$$
\psi_{\xi,\lambda}=\mathbb{O}\wedge\varphi_{\xi,\lambda}.
$$
As both $\mathbb{O}$ and $\varphi_{\xi,\lambda}$ are Lipschitz functions, $\psi_{\xi,\lambda}\in\lip(\unitsphere)$. 

\medskip

\noindent{\bfseries Step 3.} By the valuation property we have
$$
\nu(\psi_{\xi,\lambda})=\nu(\mathbb{O})+\nu(\varphi_{\xi,\lambda})-\nu(\mathbb{O}\vee\varphi_{\xi,\lambda}).
$$
Since $ \nu $ is homogeneous, $\nu(\mathbb{O})=0$. Moreover, as $\varphi_{\xi,\lambda}$ is the support function of $ \lambda D_{\xi}-\xi $,
$$
\nu(\varphi_{\xi,\lambda})=\int_{\unitsphere}p(x)dS_{n-1}(\lambda D_\xi-\xi,x)=\mu_{n-1}(\lambda D_\xi-\xi).
$$
The area measure $S_{n-1}(\lambda D_\xi-\xi)$ is the sum of two point masses of equal weight, concentrated at $\xi$ and $-\xi$, respectively. In particular it is an even measure. As $p$ is odd, we have
$$
\nu(\varphi_{\xi,\lambda})=\int_{\unitsphere}p(x)dS_{n-1}(\lambda D_\xi-\xi,x)=0.
$$
We conclude that
$$
\nu(\psi_{\xi,\lambda})=-\nu(\mathbb{O}\vee\varphi_{\xi,\lambda}).
$$
As $\varphi_{\xi,\lambda}$ is the support function of $\lambda D_\xi-\xi$ and $\mathbb{O}$ is the support function of $\{0\}$, $\varphi_{\xi,\lambda}\vee\mathbb{O}$ is the support function of 
$$
\mbox{conv}((\lambda D_\xi-\xi)\cup\{0\})=C_{\xi,\lambda}
$$
(see for instance Lemma 2.5 in \cite{ColesantiEtAlclassinvariantvaluations2020}).  

\medskip

\noindent{\bfseries Step 4.} We now fix $\lambda\ge1$, and $\bar\xi,\xi\in \unitsphere$ such that $\xi$ belongs to the support of $\psi_{\bar\xi,\lambda}$; we have
$$
|\bar\xi-\xi|<\frac{\sqrt2}{\sqrt{\lambda^2+1}}.
$$
The argument to prove this inequality, that we present for completeness, is similar to the one used in \cite{ColesantiEtAlclassinvariantvaluations2020}, Lemma 5.4. Without loss of generality, we may assume that $\bar\xi=e_1$; we write $\xi$ as $\xi=(\xi_1,\xi')$, with $\xi'\in\R^{n-1}$. We have
$$
0\ge\psi_{\bar\xi,\lambda}(\xi)=\psi_{e_1,\lambda}(\xi)=\lambda|\xi'|-\xi_1;
$$
in particular $\xi_1\ge0$. From the previous inequality and $|\xi|=1$, we obtain
$$
|\xi'|\le \frac{\sqrt{1-|\xi'|^2}}{\lambda}\quad\Rightarrow\quad|\xi'|\le\frac{1}{\sqrt{\lambda^2+1}}.
$$ 
On the other hand,
\begin{eqnarray*}
|1-\xi_1|=1-\xi_1=1-\sqrt{1-|\xi'|^2}=\frac{|\xi'|^2}{1+\sqrt{1-|\xi'|^2}}\le|\xi'|^2\le|\xi'|.
\end{eqnarray*}
The conclusion follows easily.

\medskip

\noindent{\bfseries Step 5.} Let $N\in\N$ and $\xi_1,\dots,\xi_N\in \unitsphere$ be such that 
$$
|\xi_i-\xi_j|\ge\frac{4}{\sqrt{\lambda^2+1}},\quad\forall\ i\ne j.
$$
Then
$$
\nu\left(\bigwedge_{i=1}^N \psi_{\xi_i,\lambda}\right)=\sum_{i=1}^N\nu(\psi_{\xi_i,\lambda}).
$$
The proof is based on the fact that the functions $\psi_{\xi_i,\lambda}$ have pairwise disjoint supports and on the inclusion-exclusion principle deriving from the valuation property. The argument is similar to that used in the proof of Lemma 5.7 in \cite{ColesantiEtAlclassinvariantvaluations2020}. 

\medskip

\noindent{\bfseries Step 6.} Let $\delta\in(0,1)$. For every sufficiently small $\varepsilon>0$, there exist $N$ points $\xi_1,\dots,\xi_N\in \unitsphere$ such that
$$
|\xi_i-\xi_j|\ge\varepsilon,
$$
$$
\langle\xi_i,e_1\rangle\ge\delta\quad\forall\ i=1,\dots,N,
$$
and 
$$
N\ge\frac\alpha{\varepsilon^{n-1}},
$$
where $\alpha>0$ is a suitable constant independent of $\varepsilon$. Indeed, let 
$$
A=\{\xi\in \unitsphere\colon \langle\xi,e_1\rangle\ge\delta\},
$$
and let $A'$ be the orthogonal projection of $A$ onto the hyperplane $e_1^\perp$. $A'$ is a $(n-1)$-dimensional closed ball in $e_1^\perp$, with radius $r=\sqrt{1-\delta^2}>0$ (hence, depending only on $\delta$), which contains a $(n-1)$-dimensional cube $Q$, with side length $L=\frac r{\sqrt n}$. Let
$$
M=\left[\frac L\varepsilon\right],
$$
where $ [\cdot] $ denotes the integer part. We have
$$
M\ge\frac L\varepsilon-1\ge\frac L{2\varepsilon},
$$
for sufficiently small $\varepsilon$. On the other hand, as
$$
M\le\frac L\varepsilon,
$$
$Q$ contains a cube $Q_0$ of side length $M\varepsilon$.
$Q_0$ contains in turn a ``rectangular'' grid of $N$ points $\xi'_1,\dots,\xi'_N$, with $N=M^{n-1}$, such that
$$
|\xi'_i-\xi'_j|\ge\varepsilon\quad\forall\ i\ne j.
$$ 
For every $i=1,\dots,N$, let
$$
\xi_i=\left(\sqrt{1-|\xi'_i|^2},\xi_i'\right).
$$
We clearly have
$$
\xi_i\in \unitsphere\quad\mbox{and}\quad|\xi_i-\xi_j|\ge\varepsilon,\quad\forall\, i\ne j.
$$
Moreover, as $\xi_i'\in Q_0\subset A'$, we have $\xi_i\in A$, for every $i=1,\ldots,N$. The claim is proved, as 
$$
N=M^{n-1}\ge\left(\frac L{2\varepsilon}\right)^{n-1}.
$$

\medskip

\noindent{\bfseries Step 7.}  Let $\lambda=k\in\N$ and let $\delta$ be as in Lemma \ref{lemma estimate}. Let $\xi_1,\dots,\xi_N$ be the points found in the previous step, with the choice:
$$
\varepsilon=\frac{4}{\sqrt{1+k^2}}.
$$
We define $g_k\colon \unitsphere\to\R$ as
$$
g_k=\bigwedge_{i=1}^N\psi_{\xi_i,k},
$$
and
$$
f_k=\frac1{\lambda^p} g_k,
$$ 
where $p\in(1,2)$. As the valuation $\nu$ is $(n-1)$-homogeneous,
$$
\nu(f_k)=\frac1{\lambda^{p(n-1)}}\nu(g_k).
$$
On the other hand, using step 5 and Lemma \ref{lemma estimate}, we obtain
\begin{eqnarray*}
\nu(g_k)=\nu\left(\bigwedge_{i=1}^N\psi_{\xi_i,k}\right)=\sum_{i=1}^N\nu(\psi_{\xi_i,k})\le-N\ \frac{\omega_{n-1}}4\ k^{n-3}.
\end{eqnarray*}
By the lower estimate for $N$ established in step 6,
$$
\nu(g_k)\le -C k^{2n-4},
$$
for some constant $C>0$ depending on $\delta$ and $n$, hence
$$
\nu(f_k)\le-C\ k^{2n-4-p(n-1)}.
$$
Choosing $p$ such that
$$
1<p<\frac{2n-4}{n-1}
$$
(which is possible, as $n\ge 4$), we obtain that $2n-4-p(n-1)>0$, and consequently
\begin{equation}\label{valuation diverges}
\lim_{k\to\infty}\nu(f_k)=-\infty.
\end{equation}

\medskip

\noindent{\bfseries Step 7.} We now prove that the sequence $(f_k)_k$ converges to zero in $\lip(\unitsphere)$. The argument is similar to the one used in \cite[Lemma 5.8]{ColesantiEtAlclassinvariantvaluations2020}. As the functions $\psi_{\xi_i,k}$, $i=1,\dots,N$, are non-positive and have pairwise disjoint supports, we have 
$$
\|g_k\|_{L^\infty(\unitsphere)}=\bigvee_{i=1}^N\|\psi_{\xi_i,k}\|_{L^\infty(\unitsphere)}.
$$
On the other hand, for every $\xi\in \unitsphere$, recalling that
$$
\psi_{\xi,k}=\mathbb{O}\wedge\varphi_{\xi,k}
$$
and the definition of $\varphi_{\xi,k}$, we deduce that
$$
\|\psi_{\xi_i,k}\|_{L^\infty(\unitsphere)}\le C\ k
$$
for some $C>0$ independent of $k$ and $\xi$. Hence
$$
\|f_k\|_{L^\infty(\unitsphere)}=\frac1{k^p}\ \|g_k\|_{L^\infty(\unitsphere)}\le\frac{C}{k^{p-1}}.
$$
As $p>1$, the sequence $f_k$ converges uniformly to $0$ on $\unitsphere$.

To conclude, we prove that 
$$
\lim_{k\to\infty}\lip(f_k)=0.
$$
The Lipschitz constant of the minimum of finitely many Lipschitz functions is bounded from above by the maximum of their Lipschitz constants, hence it is sufficient to prove that, for an arbitrary $\xi\in \unitsphere$,
$$
\lim_{k\to\infty}\frac1{k^p}\lip(\psi_{\xi,k})=0,
$$
uniformly with respect to $\xi$. On the other hand, by the definition of $\varphi_{\xi,k}$,
$$
\lip(\psi_{\xi,k})\le\lip(\varphi_{\xi,k})\le C\ k,
$$
where $C>0$ is independent of $\xi$. The conclusion follows from the condition $ p>1 $.

We have proved that $ (f_k)_k $ converges to $ \mathbb{O} $ in our topology, in contradiction with \eqref{valuation diverges}.
\end{proof}

We can now conclude the proof of Theorem \ref{maintheorem:decomposition}.
\begin{proof}[Proof of Theorem \ref{maintheorem:decomposition} for odd valuations]
	Let us start with the case $k=n-1$, $n\ge 4$. By Proposition \ref{proposition_missing_valuation_odd_case}, the restriction of $T:\Val_{n-1}^-(\lip(\unitsphere))\rightarrow\Val^-_{n-1}(\R^n)$ to the corresponding spaces of $\SO(n)$-finite vectors is not surjective. Thus $\Val_{n-1}^-(\lip(\unitsphere))=0$ by Corollary \ref{corollary_T_isomorphism_on_Finite_vectors}.\\
	
	For the general case, let $\mu\in\Val_k^-(\lip(\unitsphere))$, $k\ge 3$. Let $E\subset \R^n$ be a $(k+1)$-dimensional subspace with unit sphere ${\mathrm S}(E)$. Given a function $f\in\lip({\mathrm S}(E))$, we can consider its $1$-homogeneous extension to $E$, which we may pull back to a $1$-homogeneous function on $\R^n$ using the orthogonal projection onto $E$. By restricting this function to the unit sphere in $\R^n$, we obtain a well-defined map
	\begin{align*}
	i_E^*:\lip({\mathrm S}(E))\rightarrow\lip(\unitsphere).	
	\end{align*}
	It is easy to see that $i_E^*$ is continuous. We thus obtain a well-defined map
	\begin{align*}
	(i_E)_*:\Val(\lip(\unitsphere))&\rightarrow \Val(\lip({\mathrm S}(E)))\\
	\mu&\mapsto [f\mapsto \mu(i_E^*f)]
	\end{align*}
	which preserves the degree of homogeneity and parity of the valuations. Note that if $K\in\mathcal{K}(E)$ is a convex body with support function $h_K\in \lip({\mathrm S}(E))$, then $(i_E)^*h_K=h_{i_E(K)}$, where $i_E:E\rightarrow\R^n$ is the natural inclusion. By construction, $(i_E)_*\mu\in\Val_{k}(\lip({\mathrm S}(E)))$, so it is an odd, $k$-homogeneous valuation. As $\dim E=k+1\ge 4$, our previous considerations imply $(i_E)_*\mu=0$.  In particular, $0=[(i_E)_*\mu](h_K)=\mu(h_{i_E(K)})$ for all $K\in\mathcal{K}(E)$. As this holds for all $(k+1)$-dimensional subspaces $E$, $T(\mu)$ vanishes on all $(k+1)$-dimensional convex bodies. Since $ T(\mu) $ is a $k$-homogeneous valuation, Lemma \ref{lemma:InjectivitySchneiderEmbedding} implies that $T(\mu)=0$, so that $\mu=0$, as $T$ is injective.
\end{proof}
	
	\section{Characterization of smooth valuations}
	\label{section_characterization_of_smooth_valuations}
	\subsection{Preliminary considerations}
		In the previous sections, we used the fact that $\lip(\unitsphere)$ may be equipped with an operation of $\GL(n,\R)$ derived from the interpretation of functions on the unit sphere as $1$-homogeneous functions on $\R^n$. From this point of view it seems unnatural to consider these functions as being defined on the unit sphere (which depends on a choice of scalar product) rather than as $1$-homogeneous Lipschitz functions defined on $\R^n$. As in \cite{AleskerTheoryvaluationsmanifolds2006}, we will identify $1$-homogeneous functions with sections of a certain line bundle over the space of oriented lines in $\R^n$. This reformulation will simplify the proofs of Theorem \ref{maintheorem:1homSmoothValuations} and Theorem \ref{maintheorem:2homSmoothValuations}, although it is, strictly speaking, not necessary. The constructions below involve various bundles equipped with natural operations of the general linear group. These bundles can be identified with certain bundles over the unit sphere using a scalar product, but the corresponding operations of the general linear group are slightly artificial. In particular, we would have to check that all the maps defined below are compatible with respect to these operations, whereas this is rather trivial with the formulation we are going to present.\\

		To motivate the construction, let us consider an arbitrary finite-dimensional real vector space $V$ and the space of convex bodies $\K(V)$ in $V$. Then the support function $h_K$ of $K\in\mathcal{K}(V)$ is naturally a $1$-homogeneous Lipschitz function on $V^*$ given by
		\begin{align*}
			h_K(y)=\sup_{x\in K}\langle y,x\rangle\quad\text{for } y\in V^*,
		\end{align*}
		where $\langle\cdot,\cdot\rangle$ denotes the natural pairing between $V^*$ and $V$. \\
		
		Consider the space $\P$ of oriented lines in $V^*$. For $y\in V^*\setminus\{0\}$ we will denote the oriented line through $y$ by $[y]\in\P$. For $l\in\P$ we will denote by $l_+\subset V^*\setminus\{0\}$ the positive half line with respect to the orientation of $l$. Let $L$ denote the bundle over $\P$ with fiber
		\begin{align*}
			L_{l}:=\{h:l_+\rightarrow\R: h\mbox{ is } 1\text{-homogeneous}\}\quad\text{over } l\in\P.
		\end{align*}
		We may consider the support function $h_K$ as a section of $L$ by defining $h_K(l)$, for $l\in\P$, by
		\begin{align*}
			h_K(l)[y]=\sup_{x\in K}\langle y,x\rangle,\quad\text{for } y\in l_+.
		\end{align*}
		The natural operation of $\GL(V)$ on $L$ induces an operation on sections of $L$, which we will denote by $(g,f)\mapsto g\cdot f$, for $g\in\GL(V)$, and a section $f$ of $L$, defined for $l\in\P$, $y\in l_+$ by
		\begin{align*}
			(g\cdot f)[l](y):=f(g^{-1}l)[y\circ g].
		\end{align*}	
		Note that every section $f$ of $L$ may be seen as a $1$-homogeneous function on $V^*$ by considering the  $1$-homogeneous extension $\tilde{f}:V^*\rightarrow\R$ defined by 
		\begin{align*}
		\tilde{f}(y):=\begin{cases}
		f([y])(y)& y\ne 0,\\
		0 & y=0.
		\end{cases}
		\end{align*}
		Then $f\mapsto \tilde{f}$ commutes with the natural operation of $\GL(V)$ on both spaces. In particular, $h_{gK}=g\cdot h_K$ as a section of $L$, so the map $K\mapsto h_K$ is $\GL(V)$-equivariant.\\
		
		Let $\lip(\P,L)$ denote the space of all sections of $L$ such that $\tilde{f}$ is a Lipschitz function on $V^*$. Then it is easy to see that any identification $V\cong \R^n$ induces a bijection $\lip(\P,L)\cong \lip(\unitsphere)$. Note that the operation of $\GL(V)$ on $\lip(\P,L)$ coincides with the operation of $\GL(n,\R)$ on $\lip(\unitsphere)$ defined in Section \ref{section:lip_GL_operation} under the appropriate identifications. \\
		
		Let us equip $\lip(\P,L)$ with the topology of $\lip(\unitsphere)$ with respect to any identification of the two spaces induced by identifying $V\cong \R^n$. It is easy to see that the topology does not depend on the choice of identification. We will fix such a map throughout this section and will use it to define various spaces of functions without explicitly stating the identifications.\\
		
		If $f\in \lip(\P,L)$, then $\tilde{f}$ is by definition Lipschitz continuous, and in particular it is almost everywhere differentiable with bounded differential. Thus $d\tilde{f}\in L^\infty(V^*,V)$ is well-defined and $0$-homogeneous, so we may consider this as an element of $L^\infty(\P,V)$. Therefore, we obtain 
		a continuous map
		\begin{align*}
			\bar{d}:\lip(\P,L)\rightarrow L^\infty(\P,V),
		\end{align*}
		which is $\GL(V)$-equivariant with respect to the natural operation of $\GL(V)$ on both spaces. Under the identification $\lip(\P,V)\cong \lip(\unitsphere)$, this coincides with the map $\bar{\nabla}$ from Section \ref{section:lip_GL_operation}. In particular, $\bar{d}$ is continuous.\\

		Given two sections $f,h$ of $L$, $f(l)$ and $h(l)$ are both real-valued functions on $l^+$, so we can define the pointwise maximum $f\vee h$ and minimum $f\wedge h$ by setting
		\begin{align*}
		&(f\vee h)(l)[y]:=f(l)[y]\vee h(l)[y], &&(f\wedge h)(l)[y]:=f(l)[y]\wedge h(l)[y], &&&\text{for }l\in L, y\in l^+.
		\end{align*}
		
		Note that $V\cong V^{**}$ can be naturally identified with a subspace of $\lip(\P,L)$: if $x\in V$, then we consider the element $\lambda_x\in\lip(\P,L)$ defined by
		\begin{align*}
			\lambda_x(l)[y]=\langle y,x\rangle =h_{\{x\}}(l)[y],\quad\text{for }l\in\P, y\in l_+.
		\end{align*}
		Let us consider the space $\Val(\lip(\P,L))$ of all continuous valuations $\mu:\lip(\P,L)\rightarrow\C$ that are dually translation invariant, that is, satisfying
		\begin{align*}
			\mu(f+\lambda_x)=\mu(f)\quad\text{for all } f\in\lip(\P,L), x\in V.
		\end{align*}
		We equip this space with the topology of uniform convergence on compact subsets of $\lip(\P,L)$. As before, any identification of $V$ with $\R^n$ leads to a $\GL(n,\R)$-equivariant isomorphism between $\Val(\lip(\P,L))$ and $\Val(\lip(\unitsphere))$. In particular, Theorem \ref{maintheorem:decomposition} implies that we have the decomposition
		\begin{align*}
			\Val(\lip(\P,L))=\bigoplus_{k=0}^2\Val_k^+(\lip(\P,L))\oplus \Val_k^-(\lip(\P,L))
		\end{align*}
		into homogeneous even and odd valuations, defined in the obvious way. Moreover, we have the continuous, injective and $\GL(V)$-equivariant map
		\begin{align*}
			T:\Val_k^\pm(\lip(\P,L))\rightarrow& \Val_k^\pm(V)\\
			\mu\mapsto& \left[K\mapsto \mu(h_K)\right].
		\end{align*}
	\subsection{Smooth valuations of degree $1$}
		In this section, we give integral representations for the elements of the dense subspace of \emph{smooth} valuations in $\Val_1^\pm(\lip(\P,L))$. Under the identification $\lip(\P,L)\cong \lip(\unitsphere)$, this result reduces to Theorem \ref{maintheorem:1homSmoothValuations}.
		
		Let $\Dense(\P)$ denote the density bundle over $\P$, that is, the line bundle with fibers over $l\in\P$ given by (complex) Lebesgue measures on the tangent space $T_l\P$. Note that continuous sections of this bundle can naturally be integrated over $\P$. Consider the complexified line bundle $L^*_\C:=L^*\otimes \C$, where $L^*$ denotes the dual of $L$. Let $$C^\infty(\P, L_\C^*\otimes \Dense(\P)^V$$ denote the space of all smooth sections $\varphi$ of $L^*_\C\otimes\Dense(\P)$ with
		\begin{align*}
		\int_{\P}\langle\varphi, \lambda_x\rangle=0,\quad\forall x\in V,
		\end{align*}
		where $\langle\cdot,\cdot\rangle$ denotes the natural fiberwise pairing between $L^*_{\C}$ and $L$. We equip this space with the Fr\'echet topology of uniform convergence of all partial derivatives (with respect to some identification of this space with a closed subspace of $C^\infty(\unitsphere)$ induced by a scalar product).
		For $\varphi\in C^\infty(\P, L^*_{\C}\otimes \Dense(\P)^V$ we define $\Theta_1(\varphi):\lip(\P,L)\rightarrow\C$ by
		\begin{align}
			\label{eq:DefTheta1}
			\Theta_1(\varphi)[f]:=\int_{\P}\langle\varphi, f\rangle,\quad\text{for } f\in\lip(\P,L).
		\end{align}
		Then $\Theta_1(\phi)$ is a dually translation invariant valuation. If we choose a scalar product on $V$ and identify $\lip(\P,L)\cong \lip(\unitsphere)$, this corresponds to functionals of the form 
		\begin{align*}
		f\mapsto \int_{\unitsphere} \tilde{\varphi}(x)f(x)d\hm(x),\quad\text{for } f\in\lip(\unitsphere),
		\end{align*}
		where $\tilde{\varphi}\in C^\infty(\unitsphere)$ is orthogonal to the restriction of linear functions to $\unitsphere$. In particular, 
		\begin{align*}
		\left|\int_{\unitsphere} \tilde{\varphi}(x)f(x)d\hm(x)\right| \le \hm(\unitsphere)\|\tilde{\varphi}\|_\infty\cdot \|f\|_\infty,
		\end{align*}
		which implies that $\Theta_1(\varphi)$ is continuous. Moreover, given a compact set $K\subset\lip(\unitsphere)$, we can find a constant $C>0$ such that $\|f\|_\infty\le C$ for all $f\in K$ (see Proposition \ref{proposition_bounds_on_value_and_gradient}). Thus
		\begin{align*}
			\sup_{f\in K}\left|\int_{\unitsphere} \tilde{\varphi}(x)f(x)d\hm(x)\right| \le C\cdot \hm(\unitsphere)\|\tilde{\varphi}\|_\infty,
		\end{align*}
		which implies that
		\begin{align*}
			\Theta_1:C^\infty(\P, L^*_{\C}\otimes \Dense(\P)^V\rightarrow\Val_1(\lip(\P,L))
		\end{align*}
		is continuous. It is also easy to see that this map is $\GL(V)$-equivariant with respect to the natural operation on both spaces. As $C^\infty(\P, L^*_{\C}\otimes \Dense(\P)^V$ is a smooth representation, this implies that $\Theta_1(\varphi)$ is also smooth.
		\begin{theorem}
			\label{thm:Theta1Onto}
			The map 
			\begin{align*}
			\Theta_1:C^\infty(\P, L^*_{\C}\otimes \Dense(\P)^V&\rightarrow\Val_1(\lip(\P,L))^{sm}\\
			\varphi&\mapsto \left[f\mapsto \int_{\P}\langle\varphi, f\rangle\right]
			\end{align*}
			is onto.
		\end{theorem}
		\begin{proof}
			Consider the composition of $\Theta_1$ with the embedding $T:\Val_1(\lip(\P,L))\rightarrow\Val_1(V)$. As $T\circ \Theta_1$ is continuous and $\GL(V)$-equivariant, it maps $\GL(V)$-smooth vectors to $\GL(V)$-smooth vectors. Hence, we obtain a continuous (with respect to the natural topology of the Fr\'echet space $\Val_1(V)^{sm}$) $\GL(V)$-equivariant map
			\begin{align*}
			T\circ\Theta_1:C^\infty(\P, L^*_{\C}\otimes \Dense(\P)^V&\rightarrow\Val_1(V)^{sm}.
			\end{align*}
			By Proposition \ref{proposition_smooth_representation_moderate_growth}, $C^\infty(\P, L^*_{\C}\otimes \Dense(\P)^V$ is a smooth representation of moderate growth. As $\Val_1(V)^{sm}$ is the space of smooth vectors of the Banach representation $\Val(V)$, it is also of moderate growth. Moreover, it is an admissible representation by Theorem \ref{theorem:Val_mulitplicity_free}. We can thus apply the Casselman-Wallach Theorem \ref{theorem:CaseelmanWallach} to deduce that the image of $T\circ\Theta_1$ is closed in $\Val_1(V)^{sm}$. Obviously, the image intersects the subspaces $\Val_1^{\pm}(V)^{sm}$ non-trivially, so the image is also dense by Alesker's Irreducibility Theorem (see Corollary \ref{corollary_Alesker_irreducibility_theorem_smooth_case}). Thus $T\circ\Theta_1$ is onto.
			
			Now let $\mu\in\Val_1(\lip(\P,L))^{sm}$. As $T:\Val_1(\lip(\P,L))\rightarrow\Val_1(V)$ is continuous, $T(\mu)$ is a smooth valuation, so by the preceding discussion we find $\varphi\in C^\infty(\P, L^*_{\C}\otimes \Dense(\P)^V$ with $T(\mu)=T(\Theta_1(\varphi))$. Thus $\mu=\Theta_1(\varphi)$, as $T$ is injective. The claim follows.
		\end{proof}
	The proof actually shows the following.
	\begin{corollary}
		\label{corrollary:ExtensionSmoothVal1Hom}
		$T:\Val_1(\lip(\P,L))^{sm}\rightarrow \Val_1(V)^{sm}$ is an isomorphism. In particular, any valuation $\mu\in\Val_1(V)^{sm}$ extends to a smooth valuation on $\lip(\P,L)$.
	\end{corollary}

	\subsection{Smooth valuations of degree $2$}
	Let $\Sym^2(V^*)_\C$ denote the space of complex-valued symmetric bilinear forms on $V$. For a smooth  section $\varphi\in C^\infty(\P,\Sym^2(V^*)_{\C}\otimes\Dense(\P)$
	we consider the map
	\begin{align*}
		\lip(\P,L)&\rightarrow\R\\
		f&\mapsto \int_{\P}\langle\varphi,\bar{d}f\otimes\bar{d}f\rangle,
	\end{align*}
	where the brackets denote the natural pairing between $\Sym^2(V^*)_{\C}$ and $V\otimes V$. In particular, $\langle\varphi,\bar{d}f\otimes\bar{d}f\rangle$ defines an element of $L^\infty(\P,\Dense(\P)$ for every $f\in\lip(\P,L)$, which we may integrate.
	
	If we identify $V\cong \R^n$ and exploit the usual identifications, these functionals correspond to maps
	\begin{align*}
		\lip(\unitsphere)\rightarrow&\;\R\\
		f\mapsto &\int_{\unitsphere}\langle\tilde{\varphi}(x)\bar{\nabla}f(x), \bar{\nabla} f(x)\rangle d\hm(x),
	\end{align*}
	where $\tilde{\varphi}\in C^\infty(\unitsphere,\Sym^2(\C^n))$.
	\begin{lemma}
		The map
		\begin{align*}
		\Theta_2(\varphi):\lip(\P,L)&\rightarrow\R\\
		f&\mapsto \int_{\P}\langle\varphi,\bar{d}f\otimes\bar{d}f\rangle,
		\end{align*}
		defines a continuous valuation for all $\varphi\in C^\infty(\P,\Sym^2(V^*)_{\C}\otimes\Dense(\P)$.
	\end{lemma}
	\begin{proof}
		To show that these functionals are continuous, we will identify $\lip(\P,L)\cong \lip(\unitsphere)$ using the isomorphism $V\cong\R^n$. Given $f\in\lip(\unitsphere)$ and a sequence $(f_j)_j$ that $\tau$-converges to $f$, we obtain $|\bar{\nabla}f_j|\le C$ for all $j\in\N$ and some $C>0$, see Proposition \ref{prop:characterizationSequences}. By the same proposition and Lemma \ref{lemma:relation-spherical-euclidean-gradient}, $\bar{\nabla}f_j(x)$ converges to $\bar{\nabla}f(x)$ for $\hm$-a.e. $x\in \unitsphere$. The Dominated Convergence Theorem thus gives
		\begin{align*}
		&\lim\limits_{j\rightarrow\infty}\int_{\unitsphere}\langle \tilde{\varphi}(x), \bar{\nabla}f_j\otimes\bar{\nabla}f_j\rangle d\hm(x)=\int_{\unitsphere}\langle \tilde{\varphi}(x), \bar{\nabla}f\otimes\bar{\nabla}f\rangle d\hm(x).
		\end{align*}
		This implies that $\Theta_2(\varphi)$ is continuous.\\
		
		It remains to see that the functional has the valuation property. Let $f,h\in\lip(\P,L)$ and consider the open sets
		\begin{align*}
		&U:=\{l\in \P: f(l)<h(l)\}, && W:=\{l\in \P: f(l)>h(l)\}.
		\end{align*}
		If $f,h,f\vee h$ and $f\wedge h$ are differentiable at $l\in U\cup W$, then obviously
		\begin{align*}
			\bar{d}(f\wedge h)(l)&=\begin{cases}
				\bar{d}f(l) & \text{if }l\in U,\\
				\bar{d}h(l) & \text{if }l\in W,
			\end{cases}\\
			\bar{d}(f\vee h)(l)&=\begin{cases}
				\bar{d}h(l) & \text{if }l\in U,\\
				\bar{d}f(l) & \text{if }l\in W.
			\end{cases}
		\end{align*}
		In particular, these equations hold for $\hm$-a.e. $l\in U\cup W$.
		
		On $\P\setminus (U\cup W)$, $f=h=f\vee h=f\wedge h$. Applying Lemma \ref{lemma:gradient-vanishes-on-constant-sets} to all pairwise differences of these functions, with the value $c=0$, we deduce that $\bar{d}f=\bar{d}h=\bar{d}(f\vee h)=\bar{d}(f\wedge h)$ $\hm$-a.e. on $\P\setminus (U\cup W)$. We thus obtain 
		\begin{align*}
		\bar{d}f\otimes\bar{d}f+\bar{d}h\otimes\bar{d}h=\bar{d}(f\vee h)\otimes\bar{d}(f\vee h)+\bar{d}(f\wedge h)\otimes\bar{d}(f\wedge h)\quad\hm\text{-a.e. on }\P.
		\end{align*}
		This implies that $\Theta_2(\varphi)$ is a valuation on $\lip(\P,L)$.
		
	\end{proof}

	To obtain dually translation invariant valuations, we have to restrict ourselves to a subspace of $C^\infty(\P,\Sym^2(V^*)_{\C}\otimes\Dense(\P)$. Let $\mathcal{F}$ denote the subspace of all smooth sections $\varphi$ of $\Sym^2(V^*)_{\C}\otimes\Dense(\P$ that satisfy
	\begin{align*}
		&\int_{\P}\langle \varphi,\bar{d}f\otimes \bar{d}\lambda_x\rangle=0, &&\text{for all } x\in V, f\in\lip(\P,L).
	\end{align*}

	If we identify $\lip(\P,L)\cong \lip(\unitsphere)$, the functions in $\mathcal{F}$ correspond to functions $\tilde{\varphi}\in C^\infty(\unitsphere,\Sym^2(\C^n))$ with 
	\begin{align}\label{condition}
		&\int_{\unitsphere} \tilde{\varphi}(x)\bar{\nabla}f(x)d\hm(x)=0 \quad\text{ for all }f\in\lip(\unitsphere),
	\end{align}
	where we consider $\tilde{\varphi}$ as an endomorphism of $\C^n$.

	To see that $\mathcal{F}$ is non-empty, consider $\tilde{\varphi}\in C^\infty(\unitsphere,\Sym^2(\C^n))$ given by
	\begin{align*}
		\tilde{\varphi}(x)=(Id_n-xx^T)+(n-1)xx^T, \quad\mbox{for }x\in\unitsphere.
	\end{align*}
	For $f,g\in C^\infty(\unitsphere)$, integration by parts shows that 
	\begin{align*}
		\int_{\unitsphere}\langle \tilde{\varphi}(x), \bar{\nabla} f(x)\otimes \bar{\nabla}g(x)\rangle=\int_{\unitsphere} f(x)[(n-1)g(x)+\Delta g(x)]d\hm(x).
	\end{align*}
	In particular, this equation holds if $f$ is a Lipschitz section. If $g$ is the restriction of a linear function, then $(n-1)g+\Delta g=0$, so $\tilde{\varphi}$ corresponds to a non-trivial element $\varphi\in\mathcal{F}$.
	
	\begin{remark}\label{pde system} Condition \eqref{condition} can be written as a system of first order partial differential equations in the components of $\tilde{\varphi}$. After choosing a coordinate system in $\R^n$, for every $x\in\R^n$ we may represent $\tilde{\varphi}$ as a symmetric $n\times n$ matrix. Recalling the definition of $\bar{\nabla} f$, 
	$$
	\bar{\nabla}f(x)=\nabla f(x)+f(x)x,
	$$
	by the Divergence Theorem on $\unitsphere$, \eqref{condition} reads as:
	\begin{equation*}
	\int_{\unitsphere} [-f(x)\Div(\tilde{\varphi}_j(x)-\langle \tilde{\varphi}_j(x),x\rangle x)+f(x)\langle \tilde{\varphi}_j(x),x\rangle] d\hm(x)=0,
	\end{equation*}
	where $\Div$ is the divergence operator on $\unitsphere$, and, for $j\in\{1,\dots,n\}$, $\tilde{\varphi}_j$ is the $j$th column of $\tilde{\varphi}$. As $f\in\lip(\unitsphere)$ is arbitrary, we get:
	$$
	-\Div(\tilde{\varphi}_j(x)-\langle \tilde{\varphi}_j(x),x\rangle x)+\langle \tilde{\varphi}_j(x),x\rangle=0,\quad\textnormal{for a.e. }x\in\unitsphere\textnormal{ and } j=1,\ldots,n.
	$$
	\end{remark}

	Let us equip $C^\infty(\P,\Sym^2(V^*)_{\C}\otimes\Dense(\P)$ with the topology of uniform convergence of all partial derivatives (with respect to some identification of this space with $C^\infty(\unitsphere,\Sym^2(\C^n))$ induced by identifying $\R^n\cong V$), which gives $C^\infty(\P,\Sym^2(V^*)_{\C}\otimes\Dense(\P)$ the structure of a Fr\'echet space. As $V$ is a $\GL(V)$-invariant subspace of $C^\infty(\P,L)$,  the equivariance of $\bar{d}$ implies that $\mathcal{F}$ is $\GL(V)$-invariant. Moreover, it is clearly closed in the $C^\infty$-topology and thus $\mathcal{F}$ is a smooth Fr\'echet representation of $\GL(V)$, which has moderate growth by Proposition \ref{proposition_smooth_representation_moderate_growth}.\\

	\begin{lemma}
		The map \begin{align*}
		\Theta_2:\mathcal{F}&\rightarrow\Val_2(\lip(\P,L))\\
		\varphi&\mapsto \left[f\mapsto \int_{\P}\langle\varphi,\bar{d}f\otimes\bar{d}f\rangle\right]
		\end{align*}
		is well-defined and continuous.
	\end{lemma}
	\begin{proof}
		We have already seen that $\Theta_2(\varphi)$ is continuous and satisfies the valuation property for $\varphi\in\mathcal{F}$. To see that it is dually translation invariant, we calculate for $x\in V$
		\begin{align*}
			\int_{\P}\langle \varphi,\bar{d}(f+\lambda_x)\otimes \bar{d}(f+\lambda_x)\rangle=& \int_{\P}\langle \varphi,\bar{d}f\otimes \bar{d}f\rangle+ \int_{\P}\langle \varphi,\bar{d}f\otimes \bar{d}\lambda_x\rangle\\
			&+ \int_{\P}\langle \varphi,\bar{d}\lambda_x\otimes \bar{d}f\rangle+ \int_{\P}\langle \varphi,\bar{d}\lambda_x\otimes \bar{d}\lambda_x\rangle\\
			=&\int_{\P}\langle \varphi,\bar{d}f\otimes \bar{d}f\rangle,
		\end{align*}
		for $f\in C^\infty(\P,L)$, as $\varphi$ is symmetric. Since both sides are continuous, this holds for all $f\in\lip(\P,L)$, so $\Theta_2$ is well-defined.\\
			
		To see that $\Theta_2$ is continuous, let $K\subset \lip(\P,L)$ be a compact subset. By Proposition \ref{proposition_bounds_on_value_and_gradient}, there exists $L>0$ such that
		$\|\bar{d}f\|_{L^\infty(\P,V)}\le L$ for all $f\in K$. Thus
		\begin{align*}
			\|\Theta_2(\varphi)\|_K=\sup_{f\in K}|\Theta_2(\varphi)(f)|\le L^2 \|\varphi\|_{L^1(\P)},
			\end{align*}
			which shows that $\Theta_2:\mathcal{F}\rightarrow\Val_2(\lip(\P,L))$ is continuous.
		\end{proof}
	\begin{theorem}
		\label{thm:Theta2Onto}
		The map \begin{align*}
		\Theta_2:\mathcal{F}&\rightarrow\Val_2(\lip(\P,L))^{sm}\\
		\varphi&\mapsto \left[f\mapsto \int_{\P}\langle\varphi,\bar{d}f\otimes\bar{d}f\rangle\right]
		\end{align*}
		is well-defined and onto.
	\end{theorem}
	\begin{proof}
		As $\mathcal{F}$ is a smooth representation and $\Theta_2:\mathcal{F}\rightarrow\Val_2(\lip(\P,L))$ is continuous and $\GL(V)$-equivariant, the image of $\Theta_2$ is contained in $\Val_2(\lip(\P,L))^{sm}$, so this map is well-defined. 
		
		We would like to apply the same reasoning used in Theorem \ref{thm:Theta1Onto}. As a first step, we have to show that the image of $\Theta_2$ intersects $\Val^\pm_2(\lip(\P,L))$ non-trivially. To see this, we identify $\Val^\pm_2(\lip(\P,L))$ with  $\Val^\pm_2(\lip(\unitsphere))$ using an isomorphism $V\cong\R^n$. Let us start with even valuations. As discussed before, $\tilde{\varphi}\in C^\infty(\unitsphere,\Sym^2(\C^n))$ given by
		\begin{align*}
			\tilde{\varphi}(x)=(Id_n-xx^T)+(n-1)xx^T
		\end{align*} corresponds to a non-trivial section $\varphi\in\mathcal{F}$.  Moreover, 
		\begin{align*}
			\int_{\unitsphere}\langle \tilde{\varphi}(x),\bar{\nabla}f(x)\otimes\bar{\nabla}f(x) d\hm(x)=\int_{\unitsphere}\left[(n-1)f^2(x)-|\nabla f(x)|^2\right] d\hm(x),
		\end{align*}
		which is the extension of the second intrinsic volume to $\lip(\unitsphere)$ (see Theorem \ref{theorem:HadwigerLipschitz}) and thus a non-trivial even valuation.\\
		
		Let us now turn to odd valuations. As there is no odd valuation of degree $2$ for $n\le 2$ by Theorem \ref{theorem:Hadwiger}, we only have to construct an odd valuation of degree $2$ for $n\ge3$.
		
		We start by observing that the Hessian $D^2\tilde{f}$ of the $1$-homogeneous extension of $f\in C^\infty(\unitsphere)$ is invariant with respect to the addition of linear functions. The (Euclidean) Hessian of $\tilde{f}$ is related to the (spherical) Hessian $\nabla^2f$ of $f$ by
		\begin{align*}
			D^2\tilde{f}[x]|_{T_x\unitsphere}=\nabla^2 f[x]+ f(x)Id_{T_x\unitsphere},
		\end{align*}
		and $D^2\tilde{f}[x](x,\cdot)=0$ for $x\in \unitsphere$. Let $S_k$ denote the $k$th elementary symmetric polynomial in the eigenvalues of a  symmetric $(n-1)\times (n-1)$ matrix. Let $K\in\mathcal{K}(\R^n)$ be a convex body with boundary of class $C^\infty$ and strictly positive Gauss curvature at every point. Then its second area measure is, up to a normalization constant, given by
		\begin{align}\label{added}
			\psi\mapsto \int_{\unitsphere}\psi S_2([\nabla^2h_K+h_KId])d\hm(x),
		\end{align}
		for $\psi\in C^\infty(\unitsphere)$. Assume that the second area measure of $K$ is not an even measure (take, for example, a $3$-dimensional convex body which is not centrally symmetric and approximate it by smooth convex bodies). The support function $h$ of $K$ is of class $C^{\infty}(\unitsphere)$. In particular, we can choose an odd function $\psi\in C^\infty(\unitsphere)$ such that the right-hand side of \eqref{added} is not equal to zero. Fix such a function $\psi$ and let $f\in C^\infty(\unitsphere)$. Then Corollary 3.6 in \cite{ColesantiEtAlHadwigertheoremconvex2020} implies that
		\begin{eqnarray*}
			&&\frac{d}{dt}\Big|_{t=0}\int_{\unitsphere}\psi S_2([\nabla^2f+fId]+t[\nabla^2h+hId])d\hm\\
			&&=\int_{\unitsphere} \psi [dS_2](\nabla^2f+fId)[\nabla^2h+hId]d\hm\\					&&=\int_{\unitsphere} h [dS_2](\nabla^2f+fId)[\nabla^2\psi+\psi Id]d\hm,
		\end{eqnarray*}
		where $[dS_2]$ denotes the differential of $S_2$ and, for brevity, $Id=Id|_x:=Id_{T_x\unitsphere}$. As $S_2$ is a symmetric polynomial of degree $2$, this expression is symmetric in $f,h$ and $\psi$, and we thus obtain
		\begin{align*}
			&\int_{\unitsphere} \psi [dS_2](\nabla^2f+fId)[\nabla^2h+hId]d\hm=\int_{\unitsphere} h [dS_2](\nabla^2\psi+\psi Id)[\nabla^2f+fId]d\hm\\
			=&\int_{\unitsphere} h [dS_2](\nabla^2\psi+\psi Id)[\nabla^2f]d\hm+\int_{\unitsphere} hf [dS_2](\nabla^2\psi+\psi Id)[Id]d\hm\\
			=&\int_{\unitsphere} h [dS_2](\nabla^2\psi+\psi Id)[\nabla^2f]d\hm+\int_{\unitsphere} hf \mathrm{trace}[[dS_2](\nabla^2\psi+\psi Id)]d\hm.
		\end{align*}
		According to Corollary 3.5 in \cite{ColesantiEtAlHadwigertheoremconvex2020}, 
		\begin{align*}
			\int_{\unitsphere} h [dS_2](\nabla^2\psi+\psi Id)[\nabla^2f]d\hm=-\int_{\unitsphere}  [dS_2](\nabla^2\psi+\psi Id)[\nabla f\odot \nabla h]d\hm,
		\end{align*}
		where $\odot$ denotes the symmetric product, while $\mathrm{trace}[{dS_k[A]}]=[(n-1)-(k-1)]S_{k-1}(A)$ for all symmetric $(n-1)\times(n-1)$ matrices $A$, and every $k=1,\dots,n-1$ (see \cite[Chapter 1]{SalaniEquazionihessianee1999}). Thus
		\begin{align*}
			&\int_{\unitsphere} \psi [dS_2](\nabla^2f+fId)[\nabla^2h+hId]d\hm\\
			=&-\int_{\unitsphere}  [dS_2](\nabla^2\psi+\psi Id)[\nabla f\odot \nabla h]d\hm+\int_{\unitsphere} hf (n-2)[\Delta\psi+(n-1)\psi]d\hm.
		\end{align*}
		Now let $\tilde{\varphi}\in C^\infty(\unitsphere,\Sym^2(\C^n))$ be given by
		\begin{align*}
			\tilde{\varphi}(x):=&
				[-[dS_2](\nabla^2\psi(x)+\psi(x) Id|_{T_x\unitsphere})]\circ [(Id_{\R^n}-xx^T)\odot(Id_{\R^n}-xx^T)]\\
				&+(n-2)[\Delta\psi(x)+(n-1)\psi(x)]xx^T],\quad\mbox{for }x\in\unitsphere.
		\end{align*}
		Then $\tilde{\varphi}\in C^\infty(\unitsphere,\Sym^2(\C^n))$ is an odd section as $\psi$ is odd, and
		\begin{align*}
			\int_{\unitsphere}\langle \tilde{\varphi}\bar{\nabla}f,\bar{\nabla}h\rangle d\hm=&-\int_{\unitsphere}  [dS_2](\nabla^2\psi+\psi Id)[\nabla f\odot \nabla h]d\hm\\
			&+\int_{\unitsphere} fh (n-2)[\Delta\psi+(n-1)\psi]d\hm\\
			=&\int_{\unitsphere} \psi [dS_2](\nabla^2f+fId)[\nabla^2h+hId]d\hm,
		\end{align*}
		for all $f,h\in C^\infty(\unitsphere)$. Note that the functional on the right is invariant with respect to the addition of linear functions if $f,h\in C^\infty(\unitsphere)$. Thus the same holds for the left-hand side, and this extends to all $f,h\in\lip(\unitsphere)$ by continuity. We deduce that
		\begin{align*}
			\int_{\unitsphere}\langle \tilde{\varphi}\bar{\nabla}f, v\rangle d\hm=\int_{\unitsphere}\langle \tilde{\varphi}\bar{\nabla}f, \bar{\nabla}\langle v,x\rangle\rangle d\hm=0\quad \text{for all }v\in \R^n, 
		\end{align*}
		that is, $\tilde{\varphi}$ may be identified with an odd section $\varphi\in\mathcal{F}$. If $f\in C^\infty(\unitsphere)$, then
		\begin{align*}
			\Theta_2(\varphi)[f]=&\int_{\unitsphere} \psi [dS_2](\nabla^2f+fId)[\nabla^2f+fId]d\hm\\
			=&\frac{d}{dt}\Big|_{t=0}\int_{\unitsphere}\psi S_2([\nabla^2f+fId]+t[\nabla^2f+fId])d\hm\\
			=&\frac{d}{dt}\Big|_{t=0}(1+t)^2\int_{\unitsphere}\psi S_2(\nabla^2f+fId)d\hm\\
			=&2\int_{\unitsphere}\psi S_2(\nabla^2f+fId)d\hm.
		\end{align*}  
		With our choice of $\psi$,  $\Theta_2(\varphi)$ does not vanish identically on $\lip(\unitsphere)$, and so it defines a non-trivial valuation in $\Val_2^-(\lip(\unitsphere))$.\\
		
		This implies that the image of $T\circ\Theta_2:\mathcal{F}\rightarrow\Val_2(V)$ has non-trivial intersection with $\Val_2^-(V)$. We now proceed as in the proof of Theorem \ref{thm:Theta1Onto}: as $T\circ\Theta_2:\mathcal{F}\rightarrow\Val_2(V)$ is continuous and $\GL(V)$-equivariant and $\mathcal{F}$ is a smooth representation, the image of this map is contained in $\Val_2(V)^{sm}$ and $T\circ\Theta_2:\mathcal{F}\rightarrow\Val_2(V)^{sm}$ is continuous with respect to the $C^\infty$-topology on $\Val_2(V)^{sm}$. As $\mathcal{F}$ and $\Val_2(V)^{sm}$ are smooth representations of moderate growth and $\Val_2(V)^{sm}$ is admissible, the Casselman-Wallach Theorem \ref{theorem:CaseelmanWallach} implies that the image of $T\circ\Theta_2$ is closed in $\Val_2(V)^{sm}$. By the previous considerations, the image intersects $\Val_2^\pm(V)^{sm}$ non-trivially, so the image is also dense by Alesker's Irreducibility Theorem (see Corollary \ref{corollary_Alesker_irreducibility_theorem_smooth_case}). Thus $T\circ\Theta_2:\mathcal{F}\rightarrow\Val_2(V)^{sm}$ is onto. If $\mu\in\Val_2(\lip(\P,L))^{sm}$, then $T(\mu)\in \Val_2(V)^{sm}$, as $T$ is continuous, so there exists $\varphi\in \mathcal{F}$ such that $T(\mu)=T(\Theta_2(\varphi))$, which implies $\mu=\Theta_2(\varphi)$, as $T$ is injective. Thus $\Theta_2$ is onto.
	\end{proof}
	As in the $1$-homogeneous case, the proof actually shows the following.
	\begin{corollary}
		\label{corrollary:ExtensionSmoothVal2Hom}
		$T:\Val_2(\lip(\P,L))^{sm}\rightarrow \Val_2(V)^{sm}$ is an isomorphism. In particular, any valuation $\mu\in\Val_2(V)^{sm}$ extends to a smooth valuation on $\lip(\P,L)$.
	\end{corollary}	
	\bibliography{literature_Lipschitz}

	\Addresses
\end{document}